\documentclass[10pt]{article}
\usepackage{amsfonts}
\usepackage[pdftex]{graphicx}
\usepackage{amsmath}
\usepackage{amssymb}
\usepackage{float}
\usepackage{booktabs}
\usepackage[top=1in, bottom=1in, left=1in, right=1in]{geometry}
\usepackage[driverfallback=pdfmx,,colorlinks,bookmarksopen,bookmarksnumbered,citecolor=blue, urlcolor=blue]{hyperref}

\newenvironment{keywords}{
\list{}{\advance\topsep by0.35cm\relax\small
\leftmargin=1cm
\itemindent\listparindent
\rightmargin\leftmargin}\item[\hskip\labelsep
\bfseries Keywords:]}
{\endlist}
\newtheorem{theorem}{Theorem}[section]
\newtheorem{corollary}[theorem]{Corollary}
\newtheorem{lemma}[theorem]{Lemma}
\newtheorem{proposition}[theorem]{Proposition}
\newtheorem{definition}[theorem]{Definition}

\newtheorem{remark}[theorem]{Remark}
\newtheorem{proof}[theorem]{proof}

\newcommand{\dint}{\displaystyle\int}

\begin{document}

\title{Global Existence and Long Time Dynamics of a Four Compartment Brusselator Type System}

\maketitle

\centerline{\scshape Rana D. Parshad }
\medskip
{\footnotesize
 \centerline{ Department of Mathematics,}
 \centerline{Clarkson University,}
   \centerline{ Potsdam, New York 13699, USA.}

} 
\medskip

\centerline{\scshape Said Kouachi}
\medskip
{\footnotesize
   \centerline{Department of Mathematics, College of Science, }
   \centerline{Qassim University,}
   \centerline{P.O.Box 6644, Al-Gassim, Buraydah 51452, Kingdom of Saudi Arabia.}
}
\medskip

\centerline{\scshape Nitu Kumari }
\medskip
{\footnotesize
 \centerline{School of Basic Sciences,}
 \centerline{Indian Institute of Technology Mandi,}
   \centerline{ Mandi, Himachal Pradesh 175 001, India.}
   }
%

\date{}
\maketitle

\begin{abstract}
In this work we consider a four compartment Brusselator system. The reaction
terms of this system are of non constant sign, thus components of the
solution are not bounded apriori, and functional means to derive apriori
bounds will fail. We prove global existence of solutions, via construction
of an appropriate lyapunov functional. Furthermore due to the sign changing
nonlinearities, the asymptotic sign condition is also not satisfied, causing
further difficulties in proving the existence of a global attractor. These
difficulties are circumvented via the use of the lyapunov functional
constructed along with the use of the uniform gronwall Lemma. We are able to
prove the existence of an $(L^2(\Omega),H^2(\Omega))$ attractor for the
system, improving previous results in the literature from \cite{Y11b}. The Hausdorff and fractal dimensions of the attractor are also shown to be finite. In particular we derive a lower bound on the Hausdorff dimension of the global attractor. We use numerical simulations, as well as numerical attractor reconstruction methods via non linear time series analysis, to validate our results.
\end{abstract}




\section{\textbf{Introduction}}
\label{1}

The object of the current manuscript is to investigate global existence and the long time dynamics, of the following four compartment
Brusselator type reaction diffusion system.

\begin{eqnarray}
\frac{\partial u}{\partial t}-a\Delta u &=&\alpha -\left( \beta +1\right)
u+u^{2}v+D_{1}\left( w-u\right) ,\text{ \ \ \ \ \ \ \ \ \ \ \ \ \ }
\label{(1.1)} \\
\frac{\partial v}{\partial t}-b\Delta v &=&\beta u-u^{2}v+D_{2}\left(
z-v\right) ,\text{\ \ \ \ \ \ \ \ \ \ \ \ \ \ \ \ \ \ \ \ \ \ \ \ \ \ \ \ }
\label{(1.2)} \\
\frac{\partial w}{\partial t}-c\Delta w &=&\alpha -\left( \beta +1\right)
w+w^{2}z+D_{3}\left( u-w\right) ,\ \ \ \ \ \ \ \ \ \ \ \ \ \ \ \ \ 
\label{(1.3)} \\
\frac{\partial z}{\partial t}-d\Delta z &=&\beta w-w^{2}z+D_{4}\left(
v-z\right) ,\text{\ \ \ }\ \ \ \ \ \ \ \ \ \ \ \ \ \ \ \ \ \ \ \ \ \ \ \ \ \
\ \ \ \ \   \label{(1.4)}
\end{eqnarray}%
\ \ \ in $\mathbb{R}^{+}\times \Omega $, with the homogeneous Dirichlet
boundary condition%
\begin{equation}
u=v=w=z=0\ \text{ in }\ \mathbb{R}^{+}\times \partial \Omega .  \label{(1.5)}
\end{equation}%
We also impose suitable initial data

\begin{equation}
u(0,x)=u_{0}(x),\ v(0,x)=v_{0}(x),\ w(0,x)=w_{0}(x),\text{ }z(0,x)=z_{0}(x)%
\text{ \ in}\;\Omega .  \label{(1.6)}
\end{equation}%
Here $\Omega $ is an open bounded domain of class $\mathbb{C}^{1}$ in $%
\mathbb{R}^{N}$, with boundary $\partial \Omega \;.\;$The constants $a$ , $b$%
, $c$, $d,\ \alpha ,\ \beta $ and $D_{i},i=1,2,3$ and $4$ are positive. The
initial data are assumed to be nonnegative. The reaction terms denoted
respectively by $f$, $g$, $h$ and $k$ are continuously differentiable
functions on $\mathbb{R}_{+}^{4}$ satisfying $f(0,v,w,z)\geq 0$, $%
g(u,0,w,z)\geq $ $0,$ $h(u,v,0,z)\geq 0$ and $k(u,v,w,0)\geq $ $0$ for all $%
u,$ $v,$ $w,$ $z$ $\geq 0$ which imply, via invariant region methods \cite{smoller}, the positivity of the solution on its interval of existence.
Several authors as \cite{K01b}, \cite{K02}, \cite{K11} , \cite
{Parshad-Kouachi-Gutierrez} and  \cite{smoller} established global existence
for solutions of m-components systems $(m\geq 2)$ with the boundary
conditions $(1.5)$.

To the best of our knowledge the most recent work \cite{Y11b}, that considers the system \eqref{(1.1)} - \eqref{(1.5)}, does so under the following constraints 

\begin{equation}
\label{s1}
a=c, \ b=d,\ D_1=D_3, \ D_2=D_4. 
\end{equation}

This constraint is required, as the techniques in \cite{Y11b}, require addition of the equations, to derive apriori bounds on the sum of the resultants. The addition causes the problematic nonlinearities to cancel because of \eqref{s1}, which facilitates the apriori estimates.
In \cite{Y11b} Global existence, existence of a $(L^2(\Omega)(\Omega),H^1_0(\Omega))$ global attractor, and existence of an exponential attractor are all demonstrated, under the assumption via \eqref{s1}.
The primary contributions of the current work are 

\begin{itemize}
\item We \emph{remove} the assumptions via \eqref{s1}, and still derive global existence. 
\item We \emph{remove} the assumptions via \eqref{s1} and show  the existence of a finite dimensional global attractor. 
\item We \emph{improve} the regularity of the attractor, and show that it is infact an $(L^2(\Omega), H^2(\Omega))$ attractor, if the space dimension $N \leq 3$.
\item  We present a new \emph{lower} bound on the Hausdorff dimension of the attractor.
\item We use nonlinear time series analysis, to numerically estimate this lower bound. We also quantify temporal chaos in the system.
\item We perform numerical simulations to elucidate the chaotic dynamics on the attractor, for a case where \eqref{s1} does not hold, that is $D_1 \neq D_2 \neq D_3 \neq D_4$. 
\end{itemize}

\begin{remark}
We would like to point out that the case with $D_1 \neq D_2 \neq D_3 \neq D_4$, as we consider, can lead to very interesting dynamical behavior, including aperiodic dynamics or a chaotic attractor. This is validated via numerical simulations in section \ref{8}. For the numerics we use Neumann boundary conditions. Note however, the existence results hold under these boundary conditions as well. 
\end{remark}

We first present a brief reveiw of the relevant literature.
The Brusselator, in its original form, is a system of 2 ODE's that model cubic autocatalytic chemical reactions \cite{prigogine}. The diffusive Brusselator system is given by

\begin{eqnarray}
\frac{\partial u}{\partial t}-d_1\Delta u &=&a -\left( b +1\right)
u+u^{2}v,
\label{(1.1b)} \\
\frac{\partial v}{\partial t}-d_2\Delta v &=&b u-u^{2}v.
\label{(1.2b)} \\
\end{eqnarray}

This system exhibits rich dynamics, including oscillations, spatiotemporal chaos and turing instabilities \cite{pena,st}. There has been a recent interest in the global dynamics of such systems, and a number of works have appeared to this end \cite{Y09,Y11,Y11b}. Equations of this form pose various challenges from a mathematical point of view. Note, the reaction terms in \eqref{(1.1)}-\eqref{(1.5)}  do not have
a constant sign, so neither component is a priori bounded, or at least bounded in some $L^p$-space, in order to apply the well known regularizing
effect.

From the point of view of long time dynamics, an inherent difficulty in  systems of the type considered, is that the asymptotic sign condition in vector version

\begin{equation}
\limsup_{|s| \rightarrow \infty} F(s) \cdot s \leq C
\end{equation}

(where $C$ is a positive constant and $F$ is the nonlinear term, representing the reaction), is not satisfied. 
This again is primarily due to the opposite signed terms  $u^2v,w^2z$ and 
$-u^2v,-w^2z$ in equations \eqref{(1.1)} - \eqref{(1.5)}
Usually this condition plays a key role in the dissipation process, and thus if it is satisfied, often leads to 
the existence of a global attractor, for the system. This opposite signed nonlinearity however, best represents the chemical process at work. 
From a mathematical point of view, this opposite signed coupling, causes extensive problems in proving existence of a global attractor. Essentially, showing asymptotic compactness of the semigroup in question, is not straightforward anymore. This difficulty is circumvented by making various apriori estimates, where the key tool used is the uniform Gronwall lemma. 

The organisation of the current manuscript is as follows. In section \ref{2} we introduce various preliminaries that are required throught the manuscript. In section \ref{3} the global existence of strong solutions is proved via Theorem \ref{thm1} and Proposition \ref{prop1}. In section \ref{4} we prove the global existence of weak solutions, and construct absorbing sets, in the phase space $L^2(\Omega)$ via Proposition \ref{propweak} and Lemma \ref{absl2}. In section \ref{5} we make further estimates as required for the existence and regularity of the global attractor via Lemmas \ref{lemh21}, \ref{lem1h212}. In section \ref{6} we show the existence of a global attractor via Theorem \ref{t 1}. Section \ref{7} contains results on the finite dimensionality of the Hausdorff and fractal dimensions of the attractor, by providing upper bounds for them via Theorem \ref{gattrd}. Furthermore, we also present a lower bound on the Hausdorff dimension of the attractor via Theorem \ref{gattrdlower}. Section \ref{8} is devoted to numerical simulations. Here we essentially elucidate the chaotic dynamics of the attractor, for the case that $D_1 \neq D_2 \neq D_3 \neq D_4$. Lastly, in section \ref{9} we use nonlinear time series analysis to estimate lower bounds on the attractor, as well as show temporal chaos.

Also in all estimates made hence forth, $C,C_i, i=1,2,3..$ are generic constants, and can change in value from line to line, and sometimes within the same line, if so required.

\section{\textbf{Preliminary observations}}
\label{2}
The usual norms in spaces $L^{p}(\Omega)$, $L^{\infty}(\Omega)$ and $C(%
\overline{\Omega })$ are denoted respectively by

\begin{eqnarray}
\label{(2.1)} 
\left\Vert u\right\Vert _{p}^{p} &=&\int_{\Omega }\left\vert u(x)\right\vert
^{p}dx ,  \\
\left\Vert u\right\Vert _{\infty } &=&\underset{x\in \Omega }{\max }%
\left\vert u(x)\right\vert .  \label{(2.2)}
\end{eqnarray}%
It is well known that to prove global existence of solutions to \eqref{(1.1)}-\eqref{(1.5)}
\cite{henry}, it suffices to derive a uniform estimate of $%
\left\Vert f(u,v,w,z)\right\Vert _{p}$, $\left\Vert g(u,v,w,z)\right\Vert
_{p}$, $\left\Vert h(u,v,w,z)\right\Vert _{p}$ and $\left\Vert
k(u,v,w,z)\right\Vert _{p}$ on $[0;T_{max}[$ for some $p>N/2.$ Our aim is to
construct polynomial Lyapunov functionals allowing us to obtain $L^{p}-$
bounds on $u;v,w$ and $z$ $\ $that lead to global existence.
Since the functions $f,g,h$ and $k$ are continuously differentiable on $%
IR_{+}^{4}$, then for any initial data in $C(\overline{\Omega })$, it is
easy to check directly their Lipschitz continuity on bounded subsets of the
domain of a fractional power of the operator 
\begin{equation}
 \label{(2.3)}
\left( 
\begin{array}{cccc}
-a\Delta  & 0 & 0 & 0 \\ 
0 & -b\Delta  & 0 & 0 \\ 
0 & 0 & -c\Delta  & 0 \\ 
0 & 0 & 0 & -d\Delta 
\end{array}%
\right)  
\end{equation}%
Under these assumptions, the following local existence result is well known
(see \cite{pierre} .
\begin{proposition}
The system \eqref{(1.1)}-\eqref{(1.5)} admits a unique, classical solution%
\newline
$(u,v,w,z)$ on $\ (0,T_{\max }[\times \Omega $. If $T_{\max }<\infty $ then%
\newline
$\underset{t\nearrow T_{\max }}{\lim }\left\{ \left\Vert u\left( t,.\right)
\right\Vert _{\infty }+\left\Vert v\left( t,.\right) \right\Vert _{\infty
}+\left\Vert w\left( t,.\right) \right\Vert _{\infty }+\left\Vert z\left(
t,.\right) \right\Vert _{\infty }\right\} =+\infty $\newline
where $T_{\max }$ $\left\{ \left\Vert u_{0}\right\Vert _{\infty },\left\Vert
v_{0}\right\Vert _{\infty },\left\Vert w_{0}\right\Vert _{\infty
},\left\Vert z_{0}\right\Vert _{\infty }\right\} $ denotes the eventual
blow-up time.
\end{proposition}

\section{\textbf{Global existence of strong solution}}
\label{3}
\subsection{\textbf{Results}}

Put $A_{12}$ $=\frac{a+b}{2\sqrt{ab}},A_{13}$ $=\frac{a+c}{2\sqrt{ac}},A_{14}
$ $=\frac{a+d}{2\sqrt{ad}},A_{23}$ $=\frac{b+c}{2\sqrt{bc}},A_{24}$ $=\frac{%
b+d}{2\sqrt{bd}},A_{34}$ $=\frac{c+d}{2\sqrt{cd}}$. Let $\theta ,\sigma $
and $\rho $ be three positive constants such that%
\begin{equation}
\label{(3.1)}
\theta ^{2}>A_{12}^{2},  
\end{equation}%

\begin{equation}
\label{(3.2)}
\left( \theta ^{2}-A_{12}^{2}\right) \left( \sigma ^{2}-A_{23}^{2}\right)
>\left( A_{13}-A_{12}A_{23}\right) ^{2},  
\end{equation}%
and%
\begin{equation}
\label{(3.3)}
\Lambda \cdot V>\Gamma ^{2} 
\end{equation}%
where%
\begin{eqnarray*}
\Lambda  &=&\left( \theta ^{2}-A_{12}^{2}\right) \left( \sigma
^{2}-A_{23}^{2}\right) -\left( A_{13}-A_{12}A_{23}\right) ^{2} \\
V &=&\left( \theta ^{2}-A_{12}^{2}\right) \left( \sigma ^{2}\rho
^{2}-A_{24}^{2}\right) -\left( A_{14}-A_{12}A_{24}\right) ^{2} \\
\Gamma  &=&\left( \theta ^{2}-A_{12}^{2}\right) \left( A_{34}\sigma
^{2}-A_{23}A_{24}\right) -\left( A_{13}-A_{12}A_{23}\right) \left(
A_{14}-A_{12}A_{24}\right) .
\end{eqnarray*}%
Let us define three positive sequences $\left\{ \theta _{r}\right\} _{r\geq
0},\ \left\{ \sigma _{q}\right\} _{q\geq 0}$ and $\left\{ \rho _{p}\right\}
_{p\geq 0}$ satisfying 
\begin{equation*}
\dfrac{\theta _{r}\theta _{r+2}}{\theta _{r+1}^{2}}=\theta ^{2},
\end{equation*}%
\begin{equation*}
\dfrac{\sigma _{q}\sigma _{q+2}}{\sigma _{q+1}^{2}}=\sigma ^{2},
\end{equation*}%
and 
\begin{equation*}
\dfrac{\rho _{p}\rho _{p+2}}{\rho _{p+1}^{2}}=\rho ^{2},\ 0\leq r\leq q\leq
p\leq n.
\end{equation*}

\begin{remark}
We can enforce some or all of the above sequences to be increasing, and others decreasing, by choosing%
\begin{equation*}
\frac{\theta _{r+1}}{\theta _{r}}=C_{\theta }\theta ^{2r},\ \frac{\sigma
_{q+1}}{\sigma _{q}}=C_{\sigma }\sigma ^{2q}\text{ and }\frac{\rho _{p+1}}{%
\rho _{p}}=C_{\rho }\rho ^{2p},\ 0\leq r\leq q\leq p\leq n,
\end{equation*}%
where the constants $C_{\theta },$ $C_{\sigma }$ and $C_{\rho }\ $ satisfy%
\begin{equation*}
C_{\theta },C_{\sigma } \text{ and } C_{\rho }<1, \ \text{or} \ C_{\rho } > 1 \ 0\leq r\leq q\leq p\leq n.
\end{equation*}
\end{remark}

Our first main result concerning the global existence of strong solutions, of the Brusselator evolutionary system, is as follows.

\begin{theorem}
\label{thm1}
Let $\left( u\left( t,.\right) ,v\left( t,.\right) ,w\left( t,.\right)
,z\left( t,.\right) \right) $ be a solution of \eqref{(1.1)}-\eqref{(1.5)}
and let 
\begin{equation}
\label{(3.4)}
L_{n}(t)=\int_{\Omega }H_{n}\left( u\left( t,x\right) ,v\left( t,x\right)
,w\left( t,x\right) ,z\left( t,x\right) \right) dx,  
\end{equation}%
where%

\begin{equation}
\label{(3.5)}
H_{n}\left( u,v,w,z\right) =\overset{n}{\underset{p=0}{\sum }}\overset{p}{%
\underset{q=0}{\sum }}\overset{q}{\underset{r=0}{\sum }}%
C_{n}^{p}C_{p}^{q}C_{q}^{r}\theta _{r}\sigma _{q}\rho
_{p}u^{r}v^{q-r}w^{p-q}z^{n-p},  
\end{equation}%
with $n$ a positive integer.\newline
Then the functional $L_{n}$ is uniformly bounded on the interval $[0,T_{\max
}].$
\end{theorem}

\begin{corollary}
\label{cor1}
All solutions of \eqref{(1.1)}-\eqref{(1.5)} with positive initial data in $%
L^{\infty }\left( \Omega \right) $ are in $L^{\infty }$ $\left( 0,T_{\max
};L^{n}\left( \Omega \right) \right) $ for all $n\geq 1.$
\end{corollary}

\begin{proposition}
\label{prop1}
\label{global} All solutions of \eqref{(1.1)}-\eqref{(1.5)} with positive
initial data in $L^{\infty }\left( \Omega \right)$ are global.
\end{proposition}

\subsection{Proofs}

For the proof of Theorem \ref{thm1}, we need some preparatory Lemmas whose proofs are similar to those in \cite{K02}.
 Let us denote $H_{n}$ in \eqref{(3.5)} by $%
H_{n}\left( \theta _{r}\sigma _{q}\rho _{p}\right) $, then we have

\begin{lemma}
Let $H_{n}$ be the homogeneous polynomial defined by \eqref{(3.5)}. Then the first partial derivatives of $H_{n}$ are given by \newline
\begin{equation}
\left. 
\begin{array}{c}
\partial _{u}H_{n}=nH_{n-1}\left( \theta _{r+1}\sigma _{q+1}\rho
_{p+1}\right) ,\ \partial _{v}H_{n}=nH_{n-1}\left( \theta _{r}\sigma
_{q+1}\rho _{p+1}\right) , \\ 
\ \partial _{w}H_{n}=nH_{n-1}\left( \theta _{r}\sigma _{q}\rho _{p+1}\right)
,\ \partial _{z}H_{n}=nH_{n-1}\left( \theta _{r}\sigma _{q}\rho _{p}\right) 
\end{array}%
\right.   \tag{4.1}
\end{equation}%
\newline
\end{lemma}

\begin{lemma}
The second partial derivatives of $H_{n}$ are given by%
\begin{equation}
\left. 
\begin{array}{c}
\partial _{u^{2}}H_{n}=n\left( n-1\right) H_{n-2}\left( \theta _{r+2}\sigma
_{q+2}\rho _{p+2}\right) ,\ \partial _{uv}H_{n}=n\left( n-1\right)
H_{n-2}\left( \theta _{r+1}\sigma _{q+2}\rho _{p+2}\right) , \\ 
\partial _{uw}H_{n}=n\left( n-1\right) H_{n-2}\left( \theta _{r+1}\sigma
_{q+1}\rho _{p+2}\right) ,\ \partial _{uz}H_{n}=n\left( n-1\right)
H_{n-2}\left( \theta _{r+1}\sigma _{q+1}\rho _{p+1}\right) , \\ 
\partial _{v^{2}}H_{n}=n\left( n-1\right) H_{n-2}\left( \theta _{r}\sigma
_{q+2}\rho _{p+2}\right) ,\partial _{vw}H_{n}=n\left( n-1\right)
H_{n-2}\left( \theta _{r}\sigma _{q+1}\rho _{p+2}\right) , \\ 
\partial _{vz}H_{n}=n\left( n-1\right) H_{n-2}\left( \theta _{r}\sigma
_{q+1}\rho _{p+1}\right) ,\ \partial _{w^{2}}H_{n}=n\left( n-1\right)
H_{n-2}\left( \theta _{r}\sigma _{q}\rho _{p+2}\right) , \\ 
\partial _{wz}H_{n}=n\left( n-1\right) =n\left( n-1\right) H_{n-2}\theta
_{r}\sigma _{q}\rho _{p+1},\ \partial _{z^{2}}H_{n}=n\left( n-1\right)
H_{n-2}\left( \theta _{r}\sigma _{q}\rho _{p}\right) \ 
\end{array}%
\right.   \tag{4.2}
\end{equation}
\end{lemma}

\begin{lemma}
Let $A$ be the symmetric matrix defined by

\begin{equation}
A=\left[ 
\begin{array}{cccc}
a_{11} & a_{12} & a_{13} & a_{14} \\ 
a_{12} & a_{22} & a_{23} & a_{24} \\ 
a_{13} & a_{23} & a_{33} & a_{34} \\ 
a_{14} & a_{24} & a_{34} & a_{44}%
\end{array}%
\right]
\end{equation}

then 
\begin{equation*}
a_{11}^{2}\left( a_{11}a_{22}-a_{12}^{2}\right) \det A=\left( PQ-R^{2}\right)
\end{equation*}%
where 
\begin{eqnarray*}
P &=&\left( a_{11}a_{22}-a_{12}^{2}\right) \left(
a_{11}a_{33}-a_{13}^{2}\right) -\left( a_{11}a_{23}-a_{12}a_{13}\right) ^{2},
\\
Q &=&\left( a_{11}a_{22}-a_{12}^{2}\right) \left(
a_{11}a_{44}-a_{14}^{2}\right) -\left( a_{11}a_{24}-a_{12}a_{14}\right) ^{2};
\\
R &=&\left( a_{11}a_{22}-a_{12}^{2}\right) \left(
a_{11}a_{34}-a_{13}a_{14}\right) -\left( a_{11}a_{23}-a_{12}a_{13}\right)
\left( a_{11}a_{24}-a_{12}a_{14}\right) .
\end{eqnarray*}
\end{lemma}


\begin{proof}[Proof of Theorem \ref{thm1}]
Differentiating $L_{n}$ with respect to $t$ yields%
\begin{eqnarray*}
L_{n}^{\prime }(t) &=&\dint\limits_{\Omega }\left( \partial _{u}H_{n}\dfrac{%
\partial u}{\partial t}+\partial _{v}H_{n}\dfrac{\partial v}{\partial t}%
+\partial _{w}H_{n}\dfrac{\partial w}{\partial t}+\partial _{z}H_{n}\dfrac{%
\partial z}{\partial t}\right) dx \\
&=&\dint\limits_{\Omega }\left( a\partial _{u}H_{n}\Delta u+b\partial
_{v}H_{n}\Delta v+c\partial _{w}H_{n}\Delta w+d\partial _{z}H_{n}\Delta
z\right) dx \\
&&+\dint\limits_{\Omega }\left( f\partial _{u}H_{n}+g\partial
_{v}H_{n}+h\partial _{w}H_{n}+k\partial _{z}H_{n}\right) dx \\
&=&I+J.
\end{eqnarray*}

Using Green's formula and boundary conditions via \eqref{(1.5)}, and applying lemma 1 we
get%
\begin{equation}
\label{(4.18)}
I=-n\left( n-1\right) \dint\limits_{\Omega }\overset{n-2}{\underset{p=0}{%
\sum }}\overset{p}{\underset{q=0}{\sum }}\overset{q}{\underset{r=0}{\sum }%
}C_{n-2}^{p}C_{p}^{q}C_{q}^{r}\left( B_{rqp}T\right) \cdot
Tu^{r}v^{q-r}w^{p-q}z^{\left( n-2\right) -p}dx, 
\end{equation}%
where $\left\{ B_{rqp}\right\} ,\ r=\overline{0,q},q=\overline{0,p},$ $p=%
\overline{0,n-2}$ \ are the matrices defined by 
\begin{equation}
\label{mat}
B_{rqp}=\left( 
\begin{array}{cccc}
a\rho _{p+2}\sigma _{q+2}\theta _{r+2} & \frac{a+b}{2}\rho _{p+2}\sigma
_{q+2}\theta _{r+1} & \frac{a+c}{2}\rho _{p+2}\sigma _{q+1}\theta _{r+1} & 
\frac{a+d}{2}\rho _{p+1}\sigma _{q+1}\theta _{r+1} \\ 
\frac{a+b}{2}\rho _{p+2}\sigma _{q+2}\theta _{r+1} & b\rho _{p+2}\sigma
_{q+2}\theta _{r} & \frac{b+c}{2}\rho _{p+2}\sigma _{q+1}\theta _{r} & \frac{%
b+d}{2}\rho _{p+1}\sigma _{q+1}\theta _{r} \\ 
\frac{a+c}{2}\rho _{p+2}\sigma _{q+1}\theta _{r+1} & \frac{a+d}{2}\rho
_{p+1}\sigma _{q+1}\theta _{r+1} & c\rho _{p+2}\sigma _{q}\theta _{r} & 
\frac{c+d}{2}\rho _{p+1}\sigma _{q}\theta _{r} \\ 
\frac{b+c}{2}\rho _{p+2}\sigma _{q+1}\theta _{r} & \frac{b+d}{2}\rho
_{p+1}\sigma _{q+1}\theta _{r} & \frac{c+d}{2}\rho _{p+1}\sigma _{q}\theta
_{r} & d\rho _{p}\sigma _{q}\theta _{r}%
\end{array}%
\right) ,
\end{equation}%
and $T$ denotes the transposate vector%
\begin{equation*}
T=\left( \nabla u,\nabla v,\nabla w,\nabla z\right) ^{t}.
\end{equation*}%
\newline
From the Sylvester criterion \cite{M00}, each of the quadratic forms (with respect to $%
\nabla u$, $\nabla v,$ $\nabla w$ and $\nabla z$) associated with the
matrices $B_{rqp}$, $r=\overline{0,q},q=\overline{0,p},$ $p=\overline{0,n-2}$
is positive, if we prove the positivity of its main determinants $\Delta
_{rqp}^{j},\ j=1,2,3$ and $4$.
For fixed $r=\overline{0,q},q=\overline{0,p},$ and $p=\overline{0,n-2}$, we
see that%
\begin{equation*}
\Delta _{rqp}^{1}=a\rho _{p+2}\sigma _{q+2}\theta _{r+2}>0,
\end{equation*}
and condition \eqref{(3.1)} implies%
\begin{equation*}
\Delta _{rqp}^{2}=ab\rho _{p+2}^{2}\sigma _{q+2}^{2}\theta _{r+1}^{2}\left(
\theta ^{2}-A_{12}^{2}\right) >0.
\end{equation*}
We can show by elementary calculation that 
\begin{equation*}
\Delta _{rqp}^{3}=abc\rho _{p+2}^{3}\sigma _{q+2}\sigma _{q+1}^{2}\theta
_{r+1}^{2}\theta _{r}\left[ \left( \theta ^{2}-A_{12}^{2}\right) \left(
\sigma ^{2}-A_{23}^{2}\right) -\left( A_{13}-A_{12}A_{23}\right) ^{2}\right]
,
\end{equation*}
and this yields \eqref{(3.2)} $\Delta _{rqp}^{3}>0.$ For the last determinant,
we use lemma 3 to obtain 
\begin{equation*}
\left( \theta ^{2}-A_{12}^{2}\right) \Delta _{rqp}^{4}=abcd\rho
_{p+2}^{2}\rho _{p+1}^{2}\sigma _{q+1}^{4}\theta _{r+1}^{2}\theta
_{r}^{2}\left( \Lambda V-\Gamma ^{2}\right)
\end{equation*}
which implies, from \eqref{(3.3)}, $\Delta _{rqp}^{4}>0$. Consequently we have $%
I\leq 0$.

Substituting the expressions of the partial derivatives given by lemma 1 in
the second integral, yields
\begin{equation}
\label{J1}
J=n\underset{\Omega }{\int }\overset{n-1}{\underset{p=0}{\sum }}\overset{p}{%
\underset{q=0}{\sum }}\overset{q}{\underset{r=0}{\sum }}%
C_{n-1}^{p}C_{p}^{q}C_{q}^{r}\left[ \frac{\rho _{p+1}}{\rho _{p}}\frac{%
\sigma _{q+1}}{\sigma _{q}}\left( \frac{\theta _{r+1}}{\theta _{r}}%
f+g\right) +\frac{\rho _{p+1}}{\rho _{p}}h+k\right] \rho _{p}\sigma
_{q}\theta _{r}u^{r}v^{q-r}w^{p-q}z^{\left( n-1\right) -p}dx.
\end{equation}%
Since we can choose, $\frac{\theta _{r+1}}{\theta _{r}}<1$ and $\frac{\rho
_{p+1}}{\rho _{p}}<1,$ then

\begin{equation}
\label{(4.11)}
 \left( \frac{\theta _{r+1}}{\theta_{r}}-1\right) u^{2}v \leq 0 , \left( \frac{\rho _{p+1}}{\rho _{p}}-1\right)
w^{2}z \leq 0,
\end{equation}

\begin{equation*}
\frac{\rho _{p+1}}{\rho _{p}}\frac{\sigma _{q+1}}{\sigma _{q}}\left( \frac{%
\theta _{r+1}}{\theta _{r}}f+g\right) +\frac{\rho _{p+1}}{\rho _{p}}h+k\leq
C_{3}\left( u+v+w+z+1\right) ,
\end{equation*}%
then we have%
\begin{equation}
\label{Je}
J\leq C_{4}\underset{\Omega }{\int }\overset{n-1}{\underset{p=0}{\sum }}%
\overset{p}{\underset{q=0}{\sum }}\overset{q}{\underset{r=0}{\sum }}%
C_{n-1}^{p}C_{p}^{q}C_{q}^{r}u^{r}v^{q-r}w^{p-q}z^{\left( n-1\right)
-p}\left( u+v+w+z+1\right) dx.
\end{equation}

To prove that the functional $L_{n}$ is uniformly bounded on the interval $%
\left[ 0,T_{\max }\right] $, first we write 
\begin{eqnarray*}
&&\overset{n-1}{\underset{p=0}{\sum }}\overset{p}{\underset{q=0}{\sum }}%
\overset{q}{\underset{r=0}{\sum }}%
C_{n-1}^{p}C_{p}^{q}C_{q}^{r}u^{r}v^{q-r}w^{p-q}z^{\left( n-1\right)
-p}\left( u+v+w+z+1\right) \\
&=&R_{n}\left( u,v,w,z\right) +S_{n-1}\left( u,v,w,z\right) ,
\end{eqnarray*}%
\newline
where $R_{n}\left( u,v,w,z\right) $ and $S_{n-1}\left( u,v,w,z\right) $are
two homogeneous polynomials of degrees $n$ and $n-1,$ respectively. 

%

By application of Holder's inequality to the integrals \eqref{Je}, one gets the functional $L_{n}$ satisfies the differential inequality 
\begin{equation*}
L_{n}^{\prime }\left( t\right) \leq C_{5}L\left( t\right) +C_{7}L^{\frac{n-1%
}{n}}\left( t\right) ,
\end{equation*}%
which for $Z=L_{n}^{\frac{1}{n}}$ can be written as 
\begin{equation*}
nZ^{\prime }\leq C_{5}Z+C_{7}.
\end{equation*}%
A simple integration gives the uniform bound of the functional $L_{n}$ on
the interval $\left[ 0,T_{\max }\right] ;$ this ends the proof of the
theorem.\bigskip
\end{proof}


\begin{proof}[Proof of corollary \ref{cor1}]
The proof of this corollary is an immediate consequence of theorem1, the
preliminary observations and the inequality 
\begin{equation*}
\underset{\Omega }{\int }\left( u+v+w+z\right) ^{n}dx\leq C_{9}L_{n}\left(
t\right) \text{ on }\left[ 0,T^{\ast }\right[ ,
\end{equation*}%
for all $n\geq 1,$ where $C_{9}$ is a positive constant$.$
\end{proof}

\begin{proof}[Proof of proposition \ref{prop1}]
Since the degree of the polynomials in the reaction terms are three, then
from corollary 1, there exists a positive constants $C_{10}$ such that 
\begin{equation*}
\underset{\Omega }{\int }\left( u+v+w+z+1\right) ^{3n}dx\leq C_{10}\text{ on 
}\left[ 0,T_{\max }\right[ ,
\end{equation*}%
for all $n\geq 1.$The reaction terms are in $L^{\infty }\left( 0,T^{\ast
};L^{n}\left( \Omega \right) \right) $ for some $n>\dfrac{N}{2}.$Then from
the preliminary observations the solution is global.
\end{proof}

\section{Global existence of weak solutions and construction of the absorbing set }
\label{4}
\subsection{The $(v-z)$ $L^p$ absorbing set}
We first recall the following definition,

\begin{definition}[Bounded absorbing set]
A bounded set $\mathcal{B}$ in a reflexive Banach space $H$ is called a
bounded absorbing set if for each bounded subset $U$ of $H$, there is a time 
$T=T(U)$, such that $S(t)U \subset \mathcal{B}$ for all $t > T$. The number $%
T=T(U)$ is referred to as the \textbf{compactification time} for $S(t)U$.
This is essentially the time after which the semigroup compactifies. 
\end{definition}

Consider, for a positive real $p\geq 2$, the following functional%
\begin{equation}
\label{5.1}
K_{p}(t)=\int_{\Omega }\left( v^{p}+\delta z^{p}\right) dx,  
\end{equation}%
where $p\geq 2$ is a positive real. By differentiating $K_{p}$ with respect
to $t$, we get%

\begin{eqnarray*}
K_{p}^{\prime }(t) &=&-p(p-1)\int_{\Omega }\left( av^{p-2}\left\vert \nabla
v\right\vert ^{2}+c\delta z^{p-2}\left\vert \nabla z\right\vert ^{2}\right)dx \notag \\
&& +p\int_{\Omega }\left[ v^{p-1}\left( \beta u-u^{2}v-D_{2}v\ +D_{2}z\right) 
 +\delta z^{p-1}\left( \beta w-w^{2}z+D_{4}v\ -D_{4}z\right) \right] dx \notag \\
&& = I+J. \notag \\
\end{eqnarray*}%

Using Poincare inequality, we obtain%

\begin{equation*}
I\leq -C_{1}p(p-1)\int_{\Omega }\left( v^{p-2}\left\vert \nabla v\right\vert
^{2}+\delta z^{p-2}\left\vert \nabla z\right\vert ^{2}\right)
dx-2C_{2}\int_{\Omega }\left( v^{p}+\delta z^{p}\right) dx,
\end{equation*}%
where%
\begin{equation*}
C_{1}=\min \left( a,c\right) ,\ \ C_{2}=\tfrac{(p-1)}{2p}\lambda _{1}\min
\left( a,c\right)
\end{equation*}%
\begin{equation*}
J=p\int_{\Omega }\left[ v^{p-2}\left( \beta uv-u^{2}v^{2}\right) +\delta
z^{p-2}\left( \beta wz-w^{2}z^{2}\right) -\left( v-z\right) \left(
D_{2}v^{p-1}-cD_{4}z^{p-1}\right) \right] dx
\end{equation*}%
By choosing $\delta =\frac{D_{2}}{D_{4}}$, we get%
\begin{eqnarray*}
J &=&p\int_{\Omega }\left[ v^{p-2}\left( \beta uv-u^{2}v^{2}\right) +\delta
z^{p-2}\left( \beta wz-w^{2}z^{2}\right) -D_{2}\left( v-z\right) ^{2}\overset%
{p-2}{\underset{q=0}{\sum }}v^{q}z^{p-2-q}\right] dx \\
&\leq &p\tfrac{\beta ^{2}}{4}\int_{\Omega }\left[ v^{p-2}+\delta z^{p-2}%
\right] dx.
\end{eqnarray*}%

Applying Young inequality, we get%

\begin{equation*}
J\leq C_{3}+C_{2}\int_{\Omega }\left[ v^{p}+\delta z^{p}\right] dx.
\end{equation*}%

Finally, we have%

\begin{eqnarray}
\label{5.2}
&&\frac{d}{dt}\int_{\Omega }\left[ v^{p}+\delta z^{p}\right]
dx+C_{1}p(p-1)\int_{\Omega }\left( v^{p-2}\left\vert \nabla v\right\vert
^{2}+\delta z^{p-2}\left\vert \nabla z\right\vert ^{2}\right)
dx \notag \\
&& +C_{2}\int_{\Omega }\left[ v^{p}+\delta z^{p}\right] dx\leq C_{3}, 
\end{eqnarray}%
which gives%

\begin{equation}
\label{5.3}
\int_{\Omega }\left[ v^{p}+\delta z^{p}\right] dx\leq
e^{-C_{2}t}\int_{\Omega }\left[ v_{0}^{p}+\delta z_{0}^{p}\right] dx+\frac{%
C_{3}}{C_{2}}.  
\end{equation}

\begin{remark}
Note that the $(v,z)$ component of the solution to the Brusselator system gets into an $L^p$ ball, for data in $L^2$, due to the almost instantaneous regularisation of the system, we have that the weak solution becomes a strong solution that belongs to $H^1_0$, and hence to $L^p, p \leq 6$. Thus the data without loss of generality can be taken in $H^1_0$.
\end{remark}

\subsection{ The $L^2$ absorbing set}

In this subsection, we will prove the existence of an $L^{2}$ absorbing set
for the semi group of the Brusselator. We use
the functional $L_{n}$ given in section 3, for $n=2$:

\begin{equation}
\label{5.4}
L_{2}(t)=\int_{\Omega }H_{2}\left( u\left( t,x\right) ,v\left( t,x\right)
,w\left( t,x\right) ,z\left( t,x\right) \right) dx,  
\end{equation}%
where $H_{2}$ is is given by \eqref{(3.5)} for $n=2$%
\begin{equation}
\label{5.5}
\left. 
\begin{array}{c}
H_{2}\left( u,,v,w,z\right) =\theta _{0}\sigma _{0}\rho _{0}z^{2}+2\theta
_{0}\sigma _{0}\rho _{1}wz+2\theta _{0}\sigma _{1}\rho _{1}vz+2\theta
_{1}\sigma _{1}\rho _{1}uz+\theta _{0}\sigma _{0}\rho _{2}w^{2} \\ 
+2\theta _{0}\sigma _{1}\rho _{2}vw+2\theta _{1}\sigma _{1}\rho
_{2}uw+\theta _{0}\sigma _{2}\rho _{2}v^{2}+2\theta _{1}\sigma _{2}\rho
_{2}uv+\theta _{2}\sigma _{2}\rho _{2}u^{2}%
\end{array}%
\right.   
\end{equation}%
After differentiating the functional $L_{2}$ with respect to the time $t$
and following the same steps as in section 3 taking into account the
positivity of the matrix $B_{000}$ given by \eqref{mat} and which is in this case
a constant matrix, we get%
\begin{eqnarray}
\label{5.6}
&& L_{2}^{\prime }(t)+2C_{4}\int_{\Omega }\left( \left\vert \nabla u\right\vert
^{2}+\left\vert \nabla v\right\vert ^{2}+\left\vert \nabla w\right\vert
^{2}+\left\vert \nabla z\right\vert ^{2}\right) dx \notag \\
&& \leq \int_{\Omega }R\left(
u\left( t,x\right) ,v\left( t,x\right) ,w\left( t,x\right) ,z\left(
t,x\right) \right) dx,  \notag \\
\end{eqnarray}%

where $C_{4}$ is a positive constant and $R$ is a polynomial of degree two
and which the form will be determined later. Using Poincare inequality, we
obtain%
\begin{equation}
\label{5.7}
\left. 
\begin{array}{l}
L_{2}^{\prime }(t)+C_{5}\int_{\Omega }\left( u^{2}+v^{2}+w^{2}+z^{2}\right)
dx+C_{4}\int_{\Omega }\left( \left\vert \nabla u\right\vert ^{2}+\left\vert
\nabla v\right\vert ^{2}+\left\vert \nabla w\right\vert ^{2}+\left\vert
\nabla z\right\vert ^{2}\right) dx \\ 
\leq \int_{\Omega }R\left( u,v,w,z\right) dx.%
\end{array}%
\right.   
\end{equation}%
From \eqref{J1}-\eqref{Je}, the polynomial $R$ has the following form 
\begin{equation}
\label{5.8}
R=R_{1}\left( u,w\right) +R_{2}(v,z)+R_{3}\left( u,v,w,z\right) ,  
\end{equation}%
where%
\begin{equation}
\label{5.9}
\left\{ 
\begin{array}{l}
R_{1}\left( u,w\right) =\left( \alpha _{1}u^{2}+\alpha _{2}uw+\alpha
_{3}w^{2}\right) , \\ 
R_{2}(v,z)=\left( \beta _{1}v^{2}+\beta _{2}vz+\beta _{3}z^{2}\right) , \\ 
R_{3}\left( u,v,w,z\right) =\left( \gamma _{1}u+\gamma _{2}w+\gamma
_{3}\right) \left( \delta _{1}v+\delta _{2}z+\delta _{3}\right) 
\end{array}%
\right.  
\end{equation}%
and where the coefficients are given in \eqref{(4.11)}. By application of Young
inequality, we can find a positive constants $C_{6}$ and $C_{7}$ such that%
\begin{equation}
\label{5.10}
\int_{\Omega }R_{2}(v,z)dx\leq C_{6}\int_{\Omega }\left[ v^{2}+\delta z^{2}%
\right] dx,  
\end{equation}%
and 
\begin{equation}
\label{5.11}
\int_{\Omega }R_{3}\left( u,,v,w,z\right) dx\leq \frac{C_{5}}{2}\int_{\Omega
}\left( u^{2}+w^{2}\right) dx+C_{7}\int_{\Omega }\left( v^{2}+\delta
z^{2}\right) dx.  
\end{equation}%
If we prove that%
\begin{equation}
\label{5.12}
R_{1}\left( u,w\right) \leq 0,  
\end{equation}%
then \eqref{5.7} will become%

\begin{eqnarray}
\label{5.13}
&& L_{2}^{\prime }(t)+\dfrac{C_{5}}{2}\int_{\Omega }\left(
u^{2}+v^{2}+w^{2}+z^{2}\right) dx \notag \\ 
&& +C_{4}\int_{\Omega }\left( \left\vert
\nabla u\right\vert ^{2}+\left\vert \nabla v\right\vert ^{2}+\left\vert
\nabla w\right\vert ^{2}+\left\vert \nabla z\right\vert ^{2}\right) dx \notag \\
&& \leq \left( C_{6}+C_{7}\right) \int_{\Omega }\left( v^{2}+\delta z^{2}\right) dx,%
\end{eqnarray}%

then by multiplying \eqref{5.2},for $p=2,$ by $2(\frac{C_{6}+C_{7}}{C_{2}})$ and adding to \eqref{5.13}. We shall get after simplifications%

\begin{eqnarray}
\label{5.14}
&& L_{2}^{\prime }(t)+C_{8}K_{2}^{\prime }+C_{9}\int_{\Omega }\left(
u^{2}+v^{2}+w^{2}+z^{2}\right) dx \notag \\
&& +C_{10}\int_{\Omega }\left( \left\vert
\nabla u\right\vert ^{2}+\left\vert \nabla v\right\vert ^{2}+\left\vert
\nabla w\right\vert ^{2}+\left\vert \nabla z\right\vert ^{2}\right) dx \leq
C_{11}, 
\end{eqnarray}%
where the constants are positive and independent on the initial data.

To prove \eqref{5.12}, we have from \eqref{J1} for $n=2$%
\begin{equation}
\label{5.15}
\left\{ 
\begin{array}{l}
-\alpha _{1}=\rho _{1}^{\prime }\left[ \sigma _{1}^{\prime }\left( \theta
_{1}^{\prime }\left( \beta +1\right) -\beta \right) +\sigma _{1}^{\prime
}\theta _{1}^{\prime }D_{1}-D_{3}\right] ,\  \\ 
-\alpha _{2}=\rho _{1}^{\prime }\left[ \sigma _{0}^{\prime }\left( \theta
_{0}^{\prime }\left( \beta +1\right) -\beta \right) +\left( \sigma
_{0}^{\prime }\theta _{0}^{\prime }-\sigma _{1}^{\prime }\theta _{1}^{\prime
}\right) D_{1}+\left( \beta +1\right) \right] -\beta , \\ 
-\alpha _{3}=-\rho _{1}^{\prime }\left( \sigma _{0}^{\prime }\theta
_{0}^{\prime }D_{1}-D_{3}\right) +\rho _{1}^{\prime }\left( \beta +1\right)
-\beta ,%
\end{array}%
\right.  
\end{equation}%
where%
\begin{equation}
\label{5.16}
\frac{\theta _{r+1}}{\theta _{r}}=\theta _{r}^{\prime },\ \frac{\sigma _{q+1}%
}{\sigma _{q}}=\sigma _{q}^{\prime },\ \frac{\rho _{p+1}}{\rho _{p}}=\rho
_{p}^{\prime },\ r=\overline{0,q},\ q=\overline{0,p},\ p=0\text{ and }1. 
\end{equation}%
The coefficient $\alpha _{3}<0$, if%
\begin{equation}
\label{5.17}
\sigma _{0}^{\prime }\theta _{0}^{\prime }<\frac{\beta +1+D_{3}}{D_{1}}\text{
and }1>\rho _{1}^{\prime }>\frac{\beta }{-\sigma _{0}^{\prime }\theta
_{0}^{\prime }D_{1}+\left( \beta +1+D_{3}\right) }  
\end{equation}%
The inequalities \eqref{5.17} are satisfied if we choose $\theta _{0}^{\prime }$
and $\sigma _{0}^{\prime }$ satisfying%
\begin{equation}
\label{5.18}
\theta _{0}^{\prime }\sigma _{0}^{\prime }<\frac{1+D_{3}}{D_{1}}.  
\end{equation}%
Then $\alpha _{1}<0$ under the following conditions%
\begin{equation}
\label{5.19}
\theta _{1}^{\prime }>\frac{\beta }{\beta +1+D_{1}}\text{ and }\sigma
_{1}^{\prime }>\frac{D_{3}}{\theta _{1}^{\prime }\left( \beta
+1+D_{1}\right) -\beta }.  
\end{equation}%
Finally $\alpha _{2}<0$, if we choose%
\begin{equation}
\label{5.20}
\sigma _{0}^{\prime }<\frac{\beta +1}{-\theta _{0}^{\prime }\left[ \beta
+1+\left( 1-\sigma \theta \right) D_{1}\right] +\beta }
\end{equation}
and

\begin{equation}
1>\rho
_{1}^{\prime }>\frac{\beta }{\left[ \sigma _{0}^{\prime }\left\{ \theta
_{0}^{\prime }\left[ \beta +1+\left( 1-\sigma \theta \right) D_{1}\right]
-\beta \right\} +\left( \beta +1\right) \right] }.  
\end{equation}%
Inequalities \eqref{5.20} are satisfied under the following condition%
\begin{equation}
\label{5.21}
\sigma _{0}^{\prime }<\frac{1}{-\theta _{0}^{\prime }\left[ \beta +1+\left(
1-\sigma \theta \right) D_{1}\right] +\beta }.  
\end{equation}%
Since we should have $\theta _{0}^{\prime },\theta _{1}^{\prime }<1$ and $%
\rho _{0}^{\prime },\rho _{1}^{\prime }<1$ to eliminate the terms $u^{2}v$
and $w^{2}z$ and get \eqref{(4.11)}, then \eqref{5.18}, \eqref{5.19} and \eqref{5.20} are satisfied if
we can choose $\sigma _{0}^{\prime }$ sufficiently small and $\sigma
_{1}^{\prime }$ sufficiently large. But there is no conditions on $\sigma
_{0}^{\prime }$ and $\sigma _{1}^{\prime }$ only the report $\frac{\sigma
_{1}^{\prime }}{\sigma _{0}^{\prime }}=\sigma ^{2}$ which should be
sufficiently large to satisfy conditions \eqref{(3.2)} and \eqref{(3.3)}. So we can choose $%
\sigma _{0}^{\prime }$ sufficiently small and $\sigma ^{2}$ sufficiently
large to get $\sigma _{1}^{\prime }$ sufficiently large. Then \eqref{5.12} is
satisfied for appropriate constants.

To get the absorbing set from \eqref{5.14}, we should remark that from the
definition of the functionals $L_{2}$ and $K_{2}$, we can find a positive
constants $C_{12}$ and $C_{13}$ such that%
\begin{eqnarray}
\label{5.22}
&& C_{12}\int_{\Omega }\left( u^{2}+v^{2}+w^{2}+z^{2}\right) dx \nonumber \\
&& \leq L_{2}(t)+C_{8}K_{2}(t) \nonumber \\
&& \leq C_{13}\int_{\Omega }\left(u^{2}+v^{2}+w^{2}+z^{2}\right) dx,\ t>0.  \nonumber \\
\end{eqnarray}%
Then \eqref{5.14} becomes%

\begin{eqnarray}
\label{5.23}
&&\frac{d}{dt}\left[ L_{2}(t)+C_{8}K_{2}(t)\right] +C_{14}\left[
L_{2}(t)+C_{8}K_{2}(t)\right] \notag \\
&& +C_{15}\int_{\Omega }\left( \left\vert \nabla
u\right\vert ^{2}+\left\vert \nabla v\right\vert ^{2}+\left\vert \nabla
w\right\vert ^{2}+\left\vert \nabla z\right\vert ^{2}\right) dx\leq C_{16}, \notag \\
\end{eqnarray}%
where all constants are independent of the initial data.

The inequalities \eqref{5.22} and \eqref{5.23} together give%

\begin{equation}
\label{5.24}
\left\Vert \left( u,,v,w,z\right) \right\Vert _{L^{2}(\Omega )}\leq
e^{-C_{17}t}\left\Vert \left( u_{0},,v_{0},w_{0},z_{0}\right) \right\Vert
_{L^{2}(\Omega )}+C_{18},\ \ t>0,  
\end{equation}%
where the constants $C_{17}$ and $C_{18}$ are strictly positive.

\subsection{Global existence of weak solutions}

In this subsection we shall prove the global existence of the weak solutions
of the Brusselator evolutionary equation \eqref{(1.1)}-\eqref{(1.4)} for spatial dimension $
N=1,\ 2$ and $3$. We have the following standard local existence in time of
weak solutions analogous to that given in section 2 concerning the local
existence in time of strong solutions 

\begin{proposition}
For any given initial data $\left( u_{0},v_{0},w_{0},z_{0}\right) $ in $%
\left[ L^{2}(\Omega )\right] ^{4}$, there exists a unique, local weak
solution $(u,v,w,z)$ on $\ (0,T_{\max }[\times \Omega $ of the Brusselator
evolutionary equations \eqref{(1.1)}-\eqref{(1.4)} with boundary conditions \eqref{(1.5)}, which
becomes a strong solution on $\left( 0,T_{\max }\right) $. If $\ T_{\max
}<\infty $ then 

\begin{equation}
\underset{t\nearrow T_{\max }}{\lim }\left\{ \left\Vert
u\left( t,.\right) \right\Vert _{2}+\left\Vert v\left( t,.\right)
\right\Vert _{2}+\left\Vert w\left( t,.\right) \right\Vert _{2}+\left\Vert
z\left( t,.\right) \right\Vert _{2}\right\} =+\infty 
\end{equation}

where $T_{\max }$ $\left\{ \left\Vert u_{0}\right\Vert _{2},\left\Vert
v_{0}\right\Vert _{2},\left\Vert w_{0}\right\Vert _{2},\left\Vert
z_{0}\right\Vert _{2}\right\} $ denotes the eventual blow-up time.
\end{proposition}

To prove the Proposition we find estimates, similar to those presented on
the two above subsections, on the Galerkin approximate solutions of the
initial value problem \eqref{(1.1)}-\eqref{(1.6)}, for more details see  \cite{Chepyzhov and
Vishik} and \cite{Lions}. Moreover, estimation \eqref{5.24} shows that the local
solution given by the above proposition will never blow up in $\left[
L^{2}(\Omega )\right] ^{4}$ at any finite time and it exists globally. We
have

\begin{proposition}
\label{propweak}
All solutions of the Brusselator evolutionary equations \eqref{(1.1)}-\eqref{(1.4)} with
boundary conditions \eqref{(1.5)} and initial data in $\left[ L^{2}(\Omega )\right]
^{4}$ are global.
\end{proposition}

Also, from the estimate via \eqref{5.24} we see there exists a time

\begin{equation}
t_{1}=\max \left( 0,\frac{ln\left( |u(0)|_{2}^{2}+|v(0)|_{2}^{2}+|w(0)|_{2}^{2}+|z(0)|_{2}^{2}\right) }{C_{17}}\right) ,
\end{equation}
such that for times $t>t_{1}$ the following uniform estimate holds

\begin{equation}
|u|_{2}^{2}+|v|_{2}^{2}+|w|_{2}^{2}+|z|_{2}^{2} \leq \frac{1}{C_{12}}(1+C_{18}) \leq C.
\end{equation}
Here $C$ is independent of time and initial data. 
Thus we deduce the following result  

\begin{lemma}
\label{absl2}
There exits a positive constant $C$, independent of time and initial data, such that the ball%
\begin{equation*}
B_{0}=\left\{ \left( u,,v,w,z\right) \in \left[ L^{2}(\Omega )\right]
^{4}:\left\Vert \left( u,,v,w,z\right) \right\Vert _{L^{2}(\Omega )}\leq
C\right\} ,
\end{equation*}%
is an absorbing set of the semiflow associated to the Brusselator
evolutionary equations.
\end{lemma}

\section{Global attractor}
\label{5}
In this section we aim to prove the existence of a global attractor for %
\eqref{(1.1)}-\eqref{(1.5)}. We restrict our selves to spatial dimension $%
N=1,2,3$. We will use the following phase spaces, 
\begin{equation*}
H= L^{2}(\Omega)\times L^{2}(\Omega) \times L^{2}(\Omega) \times
L^{2}(\Omega).
\end{equation*}

\begin{equation*}
Y= H^{1}_{0}(\Omega)\times H^{1}_{0}(\Omega) \times H^{1}_{0}(\Omega) \times
H^{1}_{0}(\Omega),
\end{equation*}
\begin{equation*}
X= H^{2}(\Omega) \cap H^{1}_{0}(\Omega) \times H^{2}(\Omega) \cap H^{1}_{0}(\Omega) \times H^{2}(\Omega) \cap H^{1}_{0}(\Omega) \times H^{2}(\Omega) \cap H^{1}_{0}(\Omega).
\end{equation*}

Recall the following definitions

\begin{definition}
Let $\mathcal{A} \subset H^{2}(\Omega)$. Then $\mathcal{A}$ is said to be a $%
(H,X)$ global attractor if the following conditions are satisfied \newline
i) $\mathcal{A}$ is compact in $X$. \newline
ii) $\mathcal{A}$ is invariant, i.e, $S(t)\mathcal{A}=\mathcal{A}, \ t \geq 0
$\newline
iii) If $B$ is bounded in $H$ then  
\begin{equation*}
dist_{X}(S(t)B,\mathcal{A}) \rightarrow 0, \ t \rightarrow \infty.
\end{equation*}
\end{definition}

In order to prove the existence of a global attractor we are required to
show: \newline
i) There exists a bounded absorbing set in the phase space. \newline
ii) The asymptotic compactness property of the semigroup in question, \cite%
{T97}, \cite{SY02}.  \newline
These are defined next

\begin{definition}[asymptotic compactness]
The semi-group $\left\{S(t)\right\}_{t \geq 0}$ associated with a dynamical
system is said to be asymptotically compact in $H^{2}(\Omega)$ if for any $%
\left\{f_{0,n}\right\}^{\infty}_{n=1}$ bounded in $L^{2}$, and a sequence of
times $\left\{t_{n} \rightarrow \infty\right\}$ , $S(t_{n})f_{0,n}$
possesses a convergent subsequence in $H^{2}(\Omega)$.
\end{definition}

 The existence of a bounded absorbing set in the phase
space $L^2(\Omega)$ has been shown via Lemma \ref{absl2}. 

Now we can integrate \eqref{5.23} in the time interval $[t,t+1]$ to yield,

\begin{eqnarray}
&&\int_{t}^{t+1}\left(|\nabla u|_{2}^{2}+|\nabla v|_{2}^{2}+|\nabla w|_{2}^{2}+|\nabla z|_{2}^{2}\right)ds \nonumber \\
&& \leq | u(t)|_{2}^{2}+| v(t)|_{2}^{2}+| w(t)|_{2}^{2}+| z(t)|_{2}^{2}+\int_{t}^{t+1}C ds. \nonumber \\
\end{eqnarray}

Thus for $t \geq t_1$ we see that

\begin{equation}
\label{eq:nabl2}
\int_{t}^{t+1}|\nabla u|_{2}^{2}ds \leq \int_{t}^{t+1}\left(|\nabla u|_{2}^{2}+|\nabla v|_{2}^{2}+|\nabla w|_{2}^{2}+|\nabla z|_{2}^{2}\right)ds \leq C 
\end{equation}

%
%
We multiply \eqref{(1.1)} by $-\Delta u$ and integrate by parts over $\Omega $ to obtain,
\begin{eqnarray}
&& \frac{1}{2}\frac{d |\nabla u|^2_2}{d t} + a|\Delta u
|^2_2  \notag \\
& = & \left( \int_{\Omega}(\alpha-(b+1)u + u^{2}v+D_1(w-u) ) 
(-\Delta u) d\mathbf{x} \right)  \notag \\
 &\leq& \alpha\int_{\Omega} |\Delta u|d\mathbf{x} - \int_{\Omega}((b+1)|\nabla u|^2|d\mathbf{x}  + \int_{\Omega}|u^2|| v| |\Delta u|d\mathbf{x}  \notag \\
&&  + D_1\int_{\Omega}|w-u||\Delta u|d\mathbf{x}  \notag \\
& \leq& \frac{a}{4}|\Delta u|^2_2 + |\Omega|\frac{\alpha^2}{a} + \frac{a}{4}|\Delta u|^2_2 + \frac{CD_1^{2}}{a}\left(| u|^2_2+| w|^2_2\right) + \int_{\Omega}|u^2|| v| |\Delta u|d\mathbf{x}\notag \\
\end{eqnarray}

We now focus on

\begin{eqnarray}
&& \int_{\Omega}|u^2|| v| |\Delta u|d\mathbf{x} \notag \\
&\leq& \frac{1}{a}\int_{\Omega}|u|^4| v|^{2}d\mathbf{x} + \frac{a}{4}\int_{\Omega}|\Delta u|^2d\mathbf{x} \notag \\
&\leq& \frac{C}{a}\left(\int_{\Omega} (u^4)^{\frac{3}{2}}d\mathbf{x}\right)^{\frac{2}{3}}
\left(\int_{\Omega}\left(| v|^{2}\right)^{3}d\mathbf{x}\right)^{\frac{1}{3}} + \frac{a}{8}\int_{\Omega}|\Delta u|^2d\mathbf{x} \notag \\
&\leq& \frac{C}{a}\left(\int_{\Omega} (u^6)d\mathbf{x}\right)^{\frac{2}{3}}
\left(\int_{\Omega}| v|^{6}d\mathbf{x}\right)^{\frac{1}{3}} 
+ \frac{a}{8}\int_{\Omega}|\Delta u|^2d\mathbf{x} \notag \\
&\leq& \frac{C}{a}\left(|\nabla u|^{2}_{2}\right)^2 + \frac{a}{8}\int_{\Omega}|\Delta u|^2d\mathbf{x}
  \notag \\
\end{eqnarray}

 The above follows via Young's inequality with epsilon, the embedding of $H^1_0 \hookrightarrow L^6$, and the uniform estimates on the $L^6$ norm of the $v$ component via Lemma \ref{lem:Absorbing}, we obtain,

\begin{equation}
\frac{d}{dt}\left(|\nabla u|_{2}^{2}\right) + \frac{a}{2}|\Delta u|_{2}^{2}\leq C\left(|\nabla u|_{2}^{2}\right)^2  + \frac{CD_1^{2}}{a}\left(| u|^2_2+| w|^2_2\right)
\label{eq:fh21}
\end{equation}%

This yields

\begin{equation}
\frac{d}{dt}\left(|\nabla u|_{2}^{2}\right) \leq C\left(|\nabla u|_{2}^{2}\right)^2  + \frac{CD_1^{2}}{a}\left(| u|^2_2+| w|^2_2\right)
\label{eq:fn1g}
\end{equation}%

Now we recall the uniform gronwall lemma

\begin{lemma}[Uniform Gronwall Lemma]
\label{lem:gronwall} Let $\beta, \zeta$ and $h$ be nonnegative functions in $%
L^{1}_{loc}[0,\infty;\mathbb{R})$. Assume that $\beta$ is absolutely
continuous on $(0,\infty)$ and the following differential inequality is
satisfied

\begin{equation}
\frac{d \beta}{dt} \leq \zeta \beta + h, \ \mbox{for} \ t>0.
\end{equation}

If there exists a finite time $T > 0$ and some $r > 0$ such that

\begin{equation}
\int^{T+r}_{T}\zeta(\tau) d\tau \leq A, \ \int^{T+r}_{T}\beta(\tau) d\tau
\leq B, \ \mbox{and} \ \int^{T+r}_{T} h(\tau) d\tau \leq C,
\end{equation}

for any $t > T$, where $A,B$ and $C$ are some positive constants, then

\begin{equation}
\beta(t) \leq \left(\frac{B}{r}+C\right)e^{A}, \  \forall \ t > T+r.
\end{equation}
\end{lemma}

Thus the use of the uniform Gronwall lemma with $T=t_1$, $r=1$ and

\begin{equation}
\beta = | \nabla u|^2_2, \zeta = | \nabla u|^2_2, h = \frac{CD_1^{2}}{a}\left(| u|^2_2+| w|^2_2\right), 
\end{equation}
and the estimates via \eqref{eq:nabl2}, Lemma \ref{absl2} yield the following uniform estimate 

\begin{equation}
|\nabla u|_{2}^{2}\leq (C+\frac{CD_1^{2}}{a})e^C \leq C , \ \forall t \geq t_2 = t_{1} + 1.
\end{equation}%
The estimates for the gradients of the $v,w,z$ components is made similarly. We can thus state the following Lemma

\begin{lemma}
\label{lem:Absorbing} Let $u,v,w,z$ be solutions to \eqref{(1.1)}-%
\eqref{(1.5)} with $(u_{0}, v_{0}, w_{0},z_0) \in L^{2}(\Omega)$. There
exists a time $t_{2}$, and a constant $C$ independent of time and
initial data, and depending only on the parameters in \eqref{(1.1)}-%
\eqref{(1.5)}, such that for any $t > t_{2}$ the following uniform a
priori estimates hold: 

\begin{equation}
\label{h1e}
|\nabla u|^2_2 \le C, |\nabla v|^2_2 \le C, |\nabla w|^2_2 \le C, |\nabla z|^2_2 \le C,
\end{equation}
%
\end{lemma}


\subsection{Integral in time a priori $H^{2}$ estimate for u}


Now we can integrate \eqref{eq:fh21} in the time interval $[t,t+1]$,
to obtain

\begin{eqnarray}
\label{eq:inth2}
&& \int_{t}^{t+1}|\Delta u|_{2}^{2}ds  \notag \\
&& \leq |\nabla u(t)|_{2}^{2}+\int_{t}^{t+1}\left(C\left(|\nabla u|_{2}^{2}\right)^2  + \frac{CD_1^{2}}{a}\left(| u|^2_2+| w|^2_2\right)\right)ds \leq C , \ \forall t \geq t_2. \notag \\
\end{eqnarray}

This follows from Lemmas \ref{absl2} and \ref{lem:Absorbing}.

%
The same method is used on the other components $v,w,z$ to yield similar
estimates as above. Here the $t_{2}$ depends only on the $L^2$ norm of the initial data, and C is independent of time and initial data. 

\subsection{Uniform a priori $H^{2}$ estimate for u}

In this section we make higher order estimates for the components of the solutions. In these estimates it is assumed that 
$u$ and $\Delta u$ satisfy the same Dirichlet boundary conditions, which is true of all solution components. This can be shown rigorosly \cite{T97}.
We multiply Equation \eqref{(1.1)} by $\Delta^{2} u$ and integrate by parts
over $\Omega$ to obtain 
\begin{eqnarray}
&& \frac{1}{2}\frac{d |\Delta u|^2_2}{d t} + a|\nabla(\Delta u)|^2_2  \notag \\
& = & \left( \int_{\Omega}-\nabla(\alpha-(b+1)u + u^{2}v+D_1(w-u) ) \cdot
\nabla(\Delta u) d\mathbf{x} \right)  \notag \\
&=& \int_{\Omega}((b+1)\nabla u - u^2\nabla v - 2vu\nabla u - D_1(\nabla
w-\nabla u) )\cdot \nabla(\Delta u) d\mathbf{x}  \notag \\
&\leq& \int_{\Omega}((b+1)|\nabla u||\nabla(\Delta u)|d\mathbf{x}  + \int_{\Omega}|u^2||\nabla v| |\nabla(\Delta u)|d\mathbf{x}  \notag \\
&&  + 2\int_{\Omega}|v||u||\nabla u||\nabla(\Delta u)|d\mathbf{x} 
+ D_1\int_{\Omega}|\nabla w|| \nabla(\Delta u)|d\mathbf{x} +\int_{\Omega}|\nabla u|| \nabla(\Delta u)|d\mathbf{x} \notag \\
\end{eqnarray}

We now focus on

\begin{eqnarray}
&& \int_{\Omega}|u^2||\nabla v| |\nabla(\Delta u)|d\mathbf{x} \notag \\
&\leq& \frac{2}{a}\int_{\Omega}|u^4||\nabla v|^{2}d\mathbf{x} + \frac{a}{8}\int_{\Omega}|\nabla(\Delta u)|^2d\mathbf{x} \notag \\
&\leq& \frac{2}{a}\left(\int_{\Omega} (u^4)^{\frac{3}{2}}d\mathbf{x}\right)^{\frac{2}{3}}
\left(\int_{\Omega}\left(|\nabla v|^{2}\right)^{3}d\mathbf{x}\right)^{\frac{1}{3}} + \frac{a}{8}\int_{\Omega}|\nabla(\Delta u)|^2d\mathbf{x} \notag \\
&\leq& \frac{2}{a}\left(\int_{\Omega} (u^6)d\mathbf{x}\right)^{\frac{2}{3}}
\left(\int_{\Omega}|\nabla v|^{6}d\mathbf{x}\right)^{\frac{1}{3}} 
+ \frac{a}{8}\int_{\Omega}|\nabla(\Delta u)|^2d\mathbf{x} \notag \\
&\leq& C|\Delta v|^{2}_{2} + \frac{a}{8}\int_{\Omega}|\nabla(\Delta u)|^2d\mathbf{x}
  \notag \\
\end{eqnarray}

This follows via the use of Holder, Cauchy-Schwartz and Young's inequality. Along with the Sobolev embedding of $H^2 \hookrightarrow W^{1,6}$, and the uniform $L^6$ estimates via \eqref{5.3}. We next handle

\begin{eqnarray}
&& \int_{\Omega}|u||v||\nabla u| |\nabla(\Delta u)|d\mathbf{x} \notag \\
&\leq& C\int_{\Omega}|u|^2| v|^{2}|\nabla u|^2 d\mathbf{x} + \frac{a}{8}\int_{\Omega}|\nabla(\Delta u)|^2d\mathbf{x} \notag \\
&\leq& C\left(\int_{\Omega} (u^6)d\mathbf{x}\right)^{\frac{1}{3}}
\left(\int_{\Omega}| v|^{6}d\mathbf{x}\right)^{\frac{1}{3}} 
\left(\int_{\Omega}|\nabla u|^{6}d\mathbf{x}\right)^{\frac{1}{3}} 
+ \frac{a}{8}\int_{\Omega}|\nabla(\Delta u)|^2d\mathbf{x} \notag \\
&\leq& C|\Delta u|^{2}_{2} + \frac{a}{8}\int_{\Omega}|\nabla(\Delta u)|^2d\mathbf{x}
  \notag \\
\end{eqnarray}

This also follows via the use of Holder, Cauchy-Schwartz and Young's inequality. Along with the Sobolev embedding of $H^2 \hookrightarrow W^{1,6}$, and the uniform $L^6$ estimates via \eqref{5.3}.
Thus we have,

\begin{eqnarray}
&& \frac{d |\Delta u|^2_2}{d t} + a|\nabla(\Delta u)|^2_2  \notag \\
& \leq & C|\nabla u|^{2}_{2} +  \frac{a}{16}|\nabla(\Delta u)|^2_2 + C_1|\nabla w|^{2}_{2} +  \frac{a}{16}|\nabla(\Delta u)|^2_2 \notag \\
&+& C_2|\Delta v|^{2}_{2} + \frac{a}{8}|\nabla(\Delta u)|^2_2 +  C_3|\Delta u|^{2}_{2} + \frac{a}{8}|\nabla(\Delta u)|^2_2 \notag \\
\end{eqnarray}

This follows via Cauchy-Schwartz inequality and the earlier estimates.
Thus we obtain

\begin{equation}
\label{1es}
 \frac{d |\Delta u|^2_2}{d t} + \frac{3a}{8}|\nabla(\Delta u)|^2_2  \leq C|\nabla u|^{2}_{2} + C_1|\nabla w|^{2}_{2} + C_2|\Delta v|^{2}_{2} + C_3|\Delta u|^{2}_{2}
\end{equation}

Now similarly we can multiply Equation \eqref{(1.2)} by $\Delta^{2} v$ and integrate by parts
over $\Omega$ to obtain 
\begin{eqnarray}
&& \frac{1}{2}\frac{d |\Delta v|^2_2}{d t} + b|\nabla(\Delta v)|^2_2  \notag \\
& = & \left( \int_{\Omega}-\nabla(b u - u^{2}v+D_2(z-v) ) \cdot
\nabla(\Delta v) d\mathbf{x} \right)  \notag \\
&\leq& \int_{\Omega}(b|\nabla u||\nabla(\Delta v)|d\mathbf{x}  + \int_{\Omega}|u^2||\nabla v| |\nabla(\Delta v)|d\mathbf{x}  \notag \\
&&  + 2\int_{\Omega}|v||u||\nabla u||\nabla(\Delta v)|d\mathbf{x} 
+ D_1\int_{\Omega}|\nabla z|| \nabla(\Delta v)|d\mathbf{x} +\int_{\Omega}|\nabla v|| \nabla(\Delta v)|d\mathbf{x} \notag \\
\end{eqnarray}

In much the same manner as earlier we can derive

\begin{equation}
\label{2es}
 \frac{d |\Delta v|^2_2}{d t} + \frac{3b}{8}|\nabla(\Delta v)|^2_2  \leq C_4|\nabla u|^{2}_{2} +
 C_5|\nabla v|^{2}_{2} 
  + C_6|\nabla z|^{2}_{2} + C_7|\Delta v|^{2}_{2} + C_8|\Delta u|^{2}_{2} 
\end{equation}
 Now using the embedding of $H^3 \hookrightarrow H^2$, and adding up \eqref{1es} and \eqref{2es} we obtain, 

\begin{equation}  
\label{eq:1ae}
\frac{d}{dt}\left( |\Delta u|^2_2 + |\Delta v|^2_2 \right) \leq C\left(|\nabla u|^{2}_{2}+|\nabla v|^{2}_{2}+|\nabla w|^{2}_{2}+|\nabla z|^{2}_{2}\right) + C_1 \left(|\Delta u|^{2}_{2} +  |\Delta v|^{2}_{2}\right)
\end{equation}

We know via \eqref{eq:inth2} and lemma \ref{lem:Absorbing} that

\begin{equation}
\int_{t}^{t+1}|\Delta u(s)|^2_2ds < C, \int_{t}^{t+1}|\Delta v(s)|^2_2ds < C, \int_{t}^{t+1}C_1 ds < C_1, \ \forall t \geq t_2,
\end{equation}

and

\begin{equation}
\int_{t}^{t+1}(|\nabla u(s)|^2_2 + |\nabla w(s)|^2_2 + |\nabla v(s)|^2_2+ |\nabla z(s)|^2_2)ds
< C, \ \forall t \geq t_2,
\end{equation}

 Thus application of the uniform Gronwall lemma with

\begin{equation}
\beta = |\Delta u|^2_2 + |\Delta v|^2_2, \zeta = C_3, h = C_2(|\nabla u|^2_2+|\nabla
w|^2_2+|\nabla v|^2_2+|\nabla z|^2_2), r=1, T=t_2,
\end{equation}

and similar estimates for the other components, tells us there is a time $t_3=t_2+1$, such that the following lemma
is true

\begin{lemma}
\label{lemh21}  Consider \eqref{(1.1)}-\eqref{(1.5)}. For any solutions 
$u,v,w,z$ to the system, there exists a constant C independent of time and
initial data, and a time $t_{3}=t_2+1$, such that the following estimates hold
uniformly,  
\begin{align*}
| u(t)|_{H^{2}} & \leq C , \ \forall t \geq t_{3} \\
| v(t)|_{H^{2}} & \leq C , \ \forall t \geq t_{3} \\
| w(t)|_{H^{2}} & \leq C , \ \forall t \geq t_{3} \\
| z(t)|_{H^{2}} & \leq C , \ \forall t \geq t_{3} \\
\end{align*}
\end{lemma}

Thus the existence of a bounded absorbing set in $H^{2}(\Omega)$ has also
been established.


\subsection{Uniform a priori estimates for $\frac{\partial u}{\partial t}$}

From \eqref{(1.1)} via brute force we obtain

\begin{eqnarray}
&&\left| \frac{\partial u}{\partial t} \right|^{2}_{2}  \notag \\
&&= \int_{\Omega}\left( d_1\Delta u - \alpha -\left( \beta +1\right)
u+u^{2}v+D_{1}\left( w-u\right) \right)^{2} d\mathbf{x}  \notag \\
& \leq& C\left(|\Delta u|^{2}_{2} + |u|^{8}_{8} + |v|^{4}_{4} + |u|^{2}_{2}
+ |w|^{2}_{2}\right)  \notag \\
&\leq& C.
\end{eqnarray}
This follows via lemma \ref{lemh21} and the apriori $L^p$ bounds on the
solutions. Note since we have uniform $H^2$ estimates now via lemma \ref{lemh21}, the $L^8$ bound on $u$ follows via the embedding of $H^2 \hookrightarrow L^8$. Similar estimates can be derived for the $v,w,z$ components. We
can now state the following Lemma,

\begin{lemma}
\label{lem:lem l2 t}  Consider \eqref{(1.1)}-\eqref{(1.5)}. For any
solutions $u,v,w,z$ of the system there exists of a constant C independent
of time and initial data, and a time $t_{1}$ such
that the following estimates hold uniformly 
\begin{align*}
& \left| \frac{\partial u}{\partial t} \right|^{2}_{2} \leq C, \ \forall t >
t_{3} \\
& \left| \frac{\partial v}{\partial t} \right|^{2}_{2} \leq C, \ \forall t >
t_{3} \\
& \left| \frac{\partial w}{\partial t} \right|^{2}_{2} \leq C, \ \forall t >
t_{3} \\
& \left| \frac{\partial z}{\partial t} \right|^{2}_{2} \leq C, \ \forall t >
t_{3}
\end{align*}
\end{lemma}

We next make an integral in time estimate on $\frac{\partial \nabla u}{\partial t%
}$ We take the partial derivative w.r.t $t$ of \eqref{(1.1)} and multiply
the resulting equation by $\frac{\partial u}{\partial t}$ and integrate by
parts over $\Omega$ to obtain

\begin{eqnarray*}
&& \frac{1}{2}\frac{d}{d t}\left| \frac{\partial u}{\partial t}%
\right|^{2}_{2}+d_1 \left| \left(\frac{\partial \nabla u}{\partial t}%
\right)\right|^{2}_{2}  \notag \\
&\leq& \int_{\Omega}\left(-(\beta + 1 + D_1)\frac{\partial u}{\partial t}+u^2%
\frac{\partial v}{\partial t} + vu\frac{\partial u}{\partial t} + D_1\frac{%
\partial w}{\partial t} \right) \left(\frac{\partial u}{\partial t}\right)d%
\mathbf{x}  \notag \\
&\leq& C\left(\left|\frac{\partial u}{\partial t}\right|^{2}_{2}+\left|\frac{%
\partial w}{\partial t}\right|^{2}_{2} + \int_{\Omega}u^2\frac{\partial v}{%
\partial t}\frac{\partial u}{\partial t} d\mathbf{x} + \int_{\Omega}uv\left(%
\frac{\partial u}{\partial t}\right)^2 \right)d\mathbf{x}  \notag \\
& \leq & C\left(\left|\frac{\partial u}{\partial t}\right|^{2}_{2}+\left|%
\frac{\partial w}{\partial t}\right|^{2}_{2} \right) + C|u|^2_{\infty}\left|%
\frac{\partial v}{\partial t}\right|_2 \left|\frac{\partial u}{\partial t}%
\right|_2 + C|u|_{\infty}|v|_{\infty}\left|\frac{\partial u}{\partial t}\right|^2_{2}  \notag \\
&\leq& C\left(\left|\frac{\partial u}{\partial t}\right|^{2}_{2}+\left|%
\frac{\partial w}{\partial t}\right|^{2}_{2} \right) + C|\Delta u|^2_{2}\left|%
\frac{\partial v}{\partial t}\right|_2 \left|\frac{\partial u}{\partial t}%
\right|_2 + C|\Delta u|_{2}|\Delta v|_{2}|u|^2_{2}  \notag \\
&\leq& C\left(\left|\frac{\partial u}{\partial t}\right|^{2}_{2}+\left|%
\frac{\partial w}{\partial t}\right|^{2}_{2} \right) + C\left(\left|%
\frac{\partial v}{\partial t}\right|^2_2 + \left|\frac{\partial u}{\partial t}%
\right|^2_2\right) + C\left|\frac{\partial u}{\partial t}\right|^2_{2}  \notag \\
&\leq& C\left(\left|\frac{\partial u}{\partial
t}\right|^{2}_{2}+\left|\frac{\partial w}{\partial t}\right|^{2}_{2} + \left|\frac{%
\partial v}{\partial t}\right|^2_2\right) \notag \\
\end{eqnarray*}

Thus integrating the above in the time interval $[t,t+1]$, for $t \geq t_3$
we obtain

\begin{eqnarray}  
\label{1gh}
&& d_1\int^{t+1}_{t}\left| \frac{\partial \nabla u}{\partial t}%
\right|^{2}_{2}ds  \notag \\
&\leq& \left|\frac{\partial u(t)}{\partial t}\right|^{2}_{2}
 + \int^{t+1}_{t}C\left(\left|\frac{\partial u}{\partial
t}\right|^{2}_{2}+\left|\frac{\partial w}{\partial t}\right|^{2}_{2} + \left|\frac{%
\partial v}{\partial t}\right|^2_2\right)ds \notag \\
&\leq& C \notag \\
\end{eqnarray}

This follows via  lemma \ref{lemh21}, lemma \ref{lem:lem l2 t} and the embedding of $H^2(\Omega)
\hookrightarrow L^{\infty}(\Omega)$.


We will next make a uniform in time estimate for $\left|\frac{\partial
\nabla u}{\partial t}\right|^{2}_{2}$, where the previous estimate will be
used. We take the time derivative of \eqref{(1.1)}, then multiply through by 
$-\Delta \frac{\partial u}{\partial t}$ and integrate by parts over $\Omega$
to obtain

\begin{eqnarray*}
&& \frac{1}{2}\frac{d}{d t}\left| \frac{\partial \nabla u}{\partial t}%
\right|^{2}_{2}+d_1 \left|\Delta \left(\frac{\partial u}{\partial t}%
\right)\right|^{2}_{2}  \notag \\
&\leq& \int_{\Omega}(-(\beta + 1 + D_1)\frac{\partial u}{\partial t}+u^2%
\frac{\partial v}{\partial t} + vu\frac{\partial u}{\partial t} + D_1\frac{%
\partial w}{\partial t} )\left(-\Delta \left(\frac{\partial u}{\partial t}%
\right)\right)d\mathbf{x}  \notag \\
&\leq& C\int_{\Omega}(-(\beta + 1 + D_1)\frac{\partial u}{\partial t}+u^2%
\frac{\partial v}{\partial t} + vu\frac{\partial u}{\partial t} + D_1\frac{%
\partial w}{\partial t} )^2 + \frac{d_1}{2}\left|\Delta \frac{\partial u}{%
\partial t}\right|^2_2  \notag \\
&&\leq C\left(\left|\frac{\partial u}{\partial t}\right|^{2}_{2}+\left|\frac{%
\partial w}{\partial t}\right|^{2}_{2} + |u|^2_8|v|^2_8\left|\frac{\partial u%
}{\partial t}\right|^2_4+|u|^4_8\left|\frac{\partial v}{\partial t}%
\right|^2_4\right)  \notag \\
&& \leq C\left(\left|\frac{\partial u}{\partial t}\right|^{2}_{2}+\left|%
\frac{\partial w}{\partial t}\right|^{2}_{2} + |u|^2_8|v|^2_8\left|\frac{%
\partial \nabla u}{\partial t}\right|^2_2+|u|^4_8\left|\frac{\partial \nabla
v}{\partial t}\right|^2_2\right)
\end{eqnarray*}

This follows from the product rule for differentiation, Cauchy-Schwartz
inequality and the Sobolev embedding of $H^1_0(\Omega) \hookrightarrow
L^4(\Omega)$. Now using Poincaire's inequality, and the estimates via lemma \ref{lemh21}, we obtain 
\begin{eqnarray*}
&&\frac{d}{d t}\left| \frac{\partial \nabla u}{\partial t}\right|^{2}_{2} 
\notag \\
&\leq& C\left|\frac{\partial \nabla u}{\partial t}\right|^2_2
+ C\left(\left|\frac{\partial u}{\partial t}\right|^{2}_{2}+\left|\frac{%
\partial w}{\partial t}\right|^{2}_{2} +C_1\left|\frac{\partial \nabla v%
}{\partial t}\right|^2_2\right)  \notag \\
\end{eqnarray*}

We will now derive a uniform estimate for $\left|\frac{\partial \nabla u}{%
\partial t}\right|^2_2$ via the uniform Gronwall Lemma, \cite{T97}.

Now via the same methods in deriving \eqref{1gh} we can obtain

\begin{equation}
\int_{t}^{t+1}\left| \frac{\partial \nabla v}{\partial t}\right|^{2}_{2}
ds < C, \ \forall  \ t \geq t_3.
\end{equation}

%
Thus via the use of lemma \ref{lem:lem l2 t} and
 application of the uniform Gronwall lemma with

\begin{equation}
\beta = \left| \frac{\partial \nabla u(s)}{\partial t}\right|^{2}_{2}, \
\zeta = C , \ r=1, T=t_3
\end{equation}
and 
\begin{equation}
h =C\left(\left|\frac{\partial u(s)}{\partial t}\right|^{2}_{2}+\left|\frac{%
\partial w(s)}{\partial t}\right|^{2}_{2} +C_1\left|\frac{\partial
\nabla v(s)}{\partial t}\right|^2_2\right)
\end{equation}

we obtain

\begin{lemma}
\label{lem1h212}  Consider \eqref{(1.1)}-\eqref{(1.5)}. For any
solutions $u,v,w,z$ to the system, there exists a constant C independent of
time and initial data, and a time $t_{4}=t_{3}+1$, such that the following
estimates hold uniformly,  
\begin{align*}
\left|\frac{\partial \nabla u(t)}{\partial t}\right|_{L^{2}} & \leq C , \
\forall t \geq t_{4} \\
\left|\frac{\partial \nabla v(t)}{\partial t}\right|_{L^{2}} & \leq C , \
\forall t \geq t_{4} \\
\left|\frac{\partial \nabla w(t)}{\partial t}\right|_{L^{2}} & \leq C , \
\forall t \geq t_{4} \\
\left|\frac{\partial \nabla z(t)}{\partial t}\right|_{L^{2}} & \leq C , \
\forall t \geq t_{4} \\
\end{align*}
\end{lemma}
The estimates for the $v,w,z$ components follow similarly.

\section{Existence of global attractor}
\label{6}
In this section we prove the existence of a global attractor for %
\eqref{(1.1)}-\eqref{(1.5)}

\subsection{Preliminaries}

Recall the phase space $H$ introduced earlier 
\begin{equation*}
H= L^{2}(\Omega)\times L^{2}(\Omega) \times L^{2}(\Omega) \times
L^{2}(\Omega).
\end{equation*}
Also recall 
\begin{equation*}
Y= H^{1}_{0}(\Omega)\times H^{1}_{0}(\Omega) \times H^{1}_{0}(\Omega) \times
H^{1}_{0}(\Omega),
\end{equation*}
\begin{equation*}
X= H^{2}(\Omega) \cap H^{1}_{0}(\Omega) \times H^{2}(\Omega) \cap H^{1}_{0}(\Omega) \times H^{2}(\Omega) \cap H^{1}_{0}(\Omega) \times H^{2}(\Omega) \cap H^{1}_{0}(\Omega).
\end{equation*}

Also recall that if $\mathcal{A}$ is an $(H,H)$ attractor, then in order to
prove that it is a an $(H,X)$ attractor it suffices to show the existence of
a bounded absorbing set in $X$ as well as demonstrate the asymptotic
compactness of the semi-group in $X$, see \cite{T97}. We first state the
following Lemma.

\begin{lemma}
\label{lem:l2 attr} Consider the reaction diffusion system described via, %
\eqref{(1.1)}- \eqref{(1.5)}. There exists a $(H,H)$ global attractor $%
\mathcal{A}$ for this system, in space dimension $N \leq 3$, which is compact and invariant in $H$ and
attracts all bounded subsets of $H$ in the $H$ metric.
\end{lemma}

\begin{proof}
The system is well posed via proposition \ref{global}, hence there exists a well defined semi-group $\left\{S(t)\right\}_{t \geq 0}$ for initial data in $L^{2}(\Omega)$.
The existence of bounded absorbing sets in $H$ and $Y$ follow via the estimates derived in Lemma \ref{lem:Absorbing}. Furthermore  the compact Sobolev embedding of
\begin{equation*}
Y \hookrightarrow H
\end{equation*}
 yields the asymptotic compactness of the semi-group $\left\{S(t)\right\}_{t \geq 0}$ in $H$. The existence of an $(H,H)$ global attractor for the model now follows.
\end{proof}


\subsection{Asymptotic compactness of the semi-group in $X$}

We next demonstarte the asymptotic compactness of the semigroup in $X$.
Attempting this directly is quite cumbersome, as it will involve essentially
trying to make apriori $H^3(\Omega)$ estimates, and then use the compactness
of $H^3(\Omega) \hookrightarrow H^2(\Omega)$. We use a more elegant method,
where we exploit the form of the equation. We show the analysis for $u$, the
other variables follow similarly. Our strategy is to rewrite \eqref{(1.1)}
as 
\begin{equation}  \label{eq: 1ac}
a\Delta u = \frac{\partial u}{\partial t} - (\alpha -\left( \beta +1\right)
u+u^{2}v+D_{1}\left( w-u\right))
\end{equation}

We will demonstrate that every term on the right hand side of \eqref{eq: 1ac}
converges strongly in $L^{2}(\Omega)$. Thus we obtain that $\Delta u$ converges strongly in $L^{2}(\Omega)$, which will imply via elliptic
regularity the strong convergence of u in $H^{2}(\Omega)$. Since the same method yields the strong convergence in $H^{2}(\Omega)$, for the other components $v,w,z$, the asymptotic compactness in $X$
follows. We state the following Lemma 

\begin{lemma}
\label{lem:asy com}  The semi-group $\left\{S(t)\right\}_{t \geq 0}$
associated with the dynamical system \eqref{(1.1)}-\eqref{(1.5)} is
asymptotically compact in $X$.
\end{lemma}

\begin{proof}
    Let us denote $u_{n}(t)=S(t)u_{0,n}$ and $U(t_{n})=\frac{\partial u_{n}}{\partial t}|_{t=t_{n}}$.  We have that
    \begin{equation*}
       a\Delta u_{n}(t_{n}) = \left(U(t_{n}\right) - \alpha -\left( \beta +1\right)u_{n}(t_{n})+u_{n}^{2}(t_{n})v_{n}(t_{n})+D_{1}\left( w_{n}(t_{n})-u_{n}(t_{n})\right).
    \end{equation*}
    Via  lemma \ref{lem1h212} we have for $t \geq t_{4}$
    \begin{equation*}
        \left| \frac{\partial \nabla u}{\partial t}\right|_{2} \leq C.
    \end{equation*}
    so for n large enough $t_{n} \geq t_{4}$ and we obtain
    \begin{equation*}
        \left| \frac{\partial \nabla u_{n}}{\partial t} \right|_{2}\biggl|_{t=t_{n}} \leq C.
    \end{equation*}
    
    Thus for n large enough $t_{n} \geq t_{4}$ and we obtain
    \begin{equation*}
        |\nabla U_{n}|_{2} \leq C.
    \end{equation*}
    
    Also via Lemma \ref{lem:Absorbing} we have for $t \geq t_{4}$, the estimate
    \begin{equation*}
        |\nabla u|_{2} \leq C.
    \end{equation*}

    These uniform bounds allow us to extract weakly convergent subsequences,
    \begin{equation*}
        U_{n}(t_{n}) \rightarrow U \ \mbox{weakly in} \ H^{1}_{0}(\Omega).
    \end{equation*}
    \begin{equation*}
        u_{n}(t_{n}) \rightarrow u \ \mbox{weakly in} \ H^{1}_{0}(\Omega).
    \end{equation*}

    Thus from classical functional analysis theory, see \cite{T97}, and the compact embedding of $ H^{1}_{0}(\Omega) \hookrightarrow L^{2}(\Omega)$,
    
 we obtain
    \begin{equation}
     \label{eq:1fq1}
        U_{n}(t_{n}) \rightarrow U \ \mbox{strongly in} \ L^{2}(\Omega),
    \end{equation}
    \begin{equation}
     \label{eq:1fq2}
        u_{n}(t_{n}) \rightarrow u \ \mbox{strongly in} \ L^{2}(\Omega),
    \end{equation}

    Now recall the form of the truncated reaction term, 
   
   \begin{equation*}
   f_{n}(u_{n},v_n,w_n,z_n)=\alpha -\left( \beta +1\right)u_{n}(t_{n})+u_{n}^{2}(t_{n})v_{n}(t_{n})+D_{1}\left( w_{n}(t_{n})-u_{n}(t_{n})\right)
   \end{equation*}
    The convergence of the linear terms in $L^2(\Omega)$ is standard, as we have uniform $H^1_0(\Omega)$ estimates via lemma \ref{lem:Absorbing}. In order to show
    convergence of the nonlinear component of $f_n$ note
    
    \begin{eqnarray}
    && \mathop{\lim}_{n \rightarrow \infty}||u_{n}^{2}v_{n}-u^2v||_2 = \mathop{\lim}_{n \rightarrow \infty}||u_{n}^{2}v_{n}-u^{2}v_{n} + u^{2}v_{n} -u^2v||_2 \nonumber \\
    && \leq \mathop{\lim}_{n \rightarrow \infty}||u_{n}^{2}v_{n}-u^{2}v_{n}||_2 + \mathop{\lim}_{n \rightarrow \infty}||u^{2}v_{n} -u^2v||_2 \nonumber \\
    && \leq \mathop{\lim}_{n \rightarrow \infty}||v_n||_4^2||u_n+u||_4^2||u_n-u||_4^2 + \mathop{\lim}_{n \rightarrow \infty}||u||_4^2||v_n-v||_4^2 \nonumber \\
    && \leq \mathop{\lim}_{n \rightarrow \infty}C||v_n||_4^2||u_n+u||_4^2||\nabla u_n-\nabla u||^2_2 + \mathop{\lim}_{n \rightarrow \infty}C||u||_4^2||\nabla v_n-\nabla v||^2_2 \nonumber \\
    && \leq \mathop{\lim}_{n \rightarrow \infty}C||\nabla u_n-\nabla u||^2_2 + C||\nabla v_n-\nabla v||^2_2 \nonumber \\
    && \rightarrow 0 \nonumber
    \end{eqnarray}
    
    The convergence follows as we have uniform $H^2(\Omega)$ estimates via lemma \ref{lemh21}, hence strong convergence of the components in $H^1_0(\Omega)$. So we obtain
    \begin{equation}
    \label{eq:1fq}
        f_{n}(u_{n}) \rightarrow f(u) \ \mbox{strongly in} \ L^{2}(\Omega).
    \end{equation}
    
    Using the convergences via \eqref{eq:1fq1}, \eqref{eq:1fq2} and \eqref{eq:1fq} we obtain
    \begin{equation*}
        \Delta u_{n} \rightarrow \Delta u \ \mbox{strongly in} \ L^{2}(\Omega).
    \end{equation*}
    However this implies via elliptic regularity that
    \begin{equation*}
     u_{n} \rightarrow  u \ \mbox{strongly in} \ H^{2}(\Omega).
    \end{equation*}
    This proves the Lemma. We can now state the following result
\end{proof}

\begin{theorem}
\label{t 1}  Consider the coupled Brusselator system described via %
\eqref{(1.1)}-\eqref{(1.5)}.  There exists a $(H,X)$ global attractor $%
\mathbb{A}$ for this system, in space dimension $N \leq 3$, which is compact and invariant in $X$ and
attracts all bounded subsets of $H$ in the $X$ metric.
\end{theorem}

\begin{proof}
The system is well posed via proposition \ref{global}, hence there exists a well defined semi-group $\left\{S(t)\right\}_{t \geq 0}$ for initial data in $L^{2}(\Omega)$. We already have the existence of an $(H,H)$ global attractor via lemma \ref{lem:l2 attr}. The estimates derived via Lemma \ref{lem1h212} give us the existence of bounded absorbing sets in $X$ . Lemma \ref{lem:asy com} proves the asymptotic compactness of the semi-group $\left\{S(t)\right\}_{t \geq 0}$ for the dynamical system associated with \eqref{(1.1)}-\eqref{(1.5)}, in $X$. Thus the theorem is proved.
\end{proof}

\section{Finite Dimensionality of the Global Attractor}
\label{7}
In this section we show that the Hausdorff and fractal dimensions of the global attractor for the reaction diffusion system \eqref{(1.1)}-\eqref{(1.5)}, is finite. Recall

\subsection{Upper bound on the Haursdorff and fractal dimension of
the global attractor}

\begin{definition}[Fractal dimension]
Consider a subset $X$ of a Banach space $H$. If $\bar{X}$ is compact, the fractal dimension of $X$, denoted $d_{f}(X)$, is given by

\begin{equation}
\label{fd1}
d_{f}(X)= \limsup_{\epsilon \rightarrow 0}\frac{log N(X,\epsilon)}{log(\frac{1}{\epsilon})} .
\end{equation}
Here $N(X,\epsilon)$ denotes the minimum number of closed balls of radius $\epsilon$, required to cover $X$. Note $d_{f}(X)$ 
can take the value $+ \infty$.
\end{definition}

\begin{definition}[Hausdorff dimension]
Consider a subset $X$ of a Banach space $H$. If $\bar{X}$ is compact, the Hausdorff dimension of $X$, denoted $d_{H}(X)$, is given by

\begin{equation}
\label{hd1}
d_{H}(X)= \inf_{d > 0}\left\{d: \mathcal{H}^{d}(X)=0\right\} .
\end{equation}
Here

\begin{equation}
\label{md1}
\mathcal{H}^{d}(X)= \lim_{\epsilon \rightarrow 0} \mu(X,d,\epsilon).
\end{equation}

Where

\begin{equation}
\label{bd1}
 \mu(X,d,\epsilon)=\inf \left\{\sum_{i} r_{i}^{d}: r_{i} \leq \epsilon \ \mbox{and} \ X \subseteq \cup_{i}B(x_{i},r_{i})  \right\}.
\end{equation}
and $B(x_{i},r_{i})$ are balls with radius $r_{i}$.

\end{definition}

 We recall the following Lemma from \cite{T97}, which will  be useful to derive the requisite estimates.
\begin{lemma}
\label{t1a}
If there is an integer n such that $q_{n} < 0$ then the Hausdorff  and fractal dimensions  of $\mathcal{A}$, denoted $d_{H}(\mathcal{A})$ and $d_{F}(\mathcal{A})$, satisfy the following estimates
\begin{equation*}
d_{H}(\mathcal{A}) \leq n  
\end{equation*}

\begin{equation*}
\ d_{F}(\mathcal{A}) \leq 2n
\end{equation*}

\end{lemma}

We will provide upper bounds on these dimensions in terms of parameters in the model. There is a standard methodology to derive these estimates. We consider a volume element in the phase space, and try and derive conditions that will cause it to decay, as time goes forward. If $\mathcal{A}$ is the global attractor of the semigroup $\left\{S(t)\right\}_{t \geq 0}$ in $H$ associated with \eqref{(1.1)}-\eqref{(1.5)} , we can define

\begin{equation}\label{qn1}
 q_{n}(t) = \mathop{\sup}\limits_{u_{0}\in  A}\mathop{\sup}\limits_{g_{i} \in H, ||g_{i}||=1,1\leq i \leq n}\frac{1}{t}\int^{t}_{0}Tr(\Delta U(\tau)+F^{'}(S(\tau)u_{0})\circ Q_{n}(\tau)d\tau
\end{equation}

where

\begin{equation}
q_{n}=\mathop{\limsup}_{t \rightarrow \infty}q_{n}(t)
\end{equation} 

Here F is the nonlinear map in \eqref{(1.1)}-\eqref{(1.5)}.
also $Q_{n}$ is the orthogonal projection of the phase space $H$ onto the subspace spanned by $U_{1}(t),U_{2}(t),\cdots, U_{n}(t)$, with 
\begin{equation*}
U_{i}(t)=L(S(t)u_{0})g_{i}, i=1,2,..n.
\end{equation*}
$L(S(t)u_{0})$ is the Frechet derivative of the map $S(t)$ at $u_{0}$, with t fixed.
Also for this model, $L(S(t)u_{0})g=U(t)=(U,V,W,Z)$, where $u=(u,v,w,z)$ is a solution to \eqref{(1.1)}-\eqref{(1.5)}, $\phi_{j}=(\phi^{1}_{j}...\phi^{3}_{j},)$ are an orthonormal basis for the subspace $Q_{n}(\tau)H$ and $(U,V,W,Z)$ are strong solutions to the variational equations for the reaction diffusion system \eqref{(1.1)}-\eqref{(1.5)} whose exact form is found in \cite{Y11b}.
The estimate on the trace follows by standard means \cite{T97}, and we can now state the following result as a direct application of lemma \ref{t1a},

\begin{theorem}
\label{gattrd}
Consider the reaction diffusion equation described via, \eqref{(1.1)}-\eqref{(1.5)}. The global attractor $\mathcal{A}$ of the system, for spatial dimension $N \leq 3$, is of finite dimension.
Furthermore, explicit upper bounds for the attractors Hausdorff and fractal dimensions, are given as follows

\begin{equation}
d_{H}(A) \leq \left(\frac{C(D_i,a,b,c,d,\alpha,\beta)}{ K_{1}} \right)^{\frac{3}{2}}|\Omega| +1 
\end{equation}

\begin{equation}
 d_{F}(A) \leq 2\left(\frac{(C(D_i,a,b,c,d,\alpha,\beta)}{ K_{1}} \right)^{\frac{3}{2}}|\Omega| +2 
\end{equation}

\end{theorem}

\subsection{Lower bound on the haursdorff dimension of
the global attractor}

We consider the reaction diffusion equation described via, \eqref{(1.1)}-\eqref{(1.5)}.
A spatially uniform stationary solution is%
\begin{equation}
\label{1.7}
u\equiv w\equiv \alpha ,\ \ v\equiv z\equiv \frac{\beta }{\alpha }. 
\end{equation}%
If $A$ is the attractor of the brusselator and $L\left( t\right) $ is the
frechet derivative of the semigroup associated to the brusselator, then%
\begin{equation*}
\dim _{H}A\geq n,
\end{equation*}%
where $n$ is the first integer satisfying%
\begin{equation*}
\emph{Re}\lambda_{n}>0>\emph{Re}\lambda_{n+1},
\end{equation*}%
where $\left\{ \lambda _{n}\right\} $ represents the set of the eigenvalues
of $L\left( t\right) $. If we put%
\begin{equation*}
L\left( t\right) \left( U_{0},V_{0},W_{0},Z_{0}\right) =\left(
U,V,W,Z\right) ,\ \ t>0,
\end{equation*}%
\begin{eqnarray}
\frac{\partial U}{\partial t}-a\Delta U &=&\left( -\left( \beta
+1+D_{1}\right) +2uv\right) U\ \ +u^{2}V\ +D_{1}W,\text{ \ \ \ \ \ \ \ \ \ \
\ \ }   \\
\frac{\partial V}{\partial t}-b\Delta V &=&\left( \beta -2uv\right) U\
-\left( D_{2}+u^{2}\right) V\ \ \ +D_{2}Z,\text{\ \ \ \ \ \ \ \ \ \ \ \ \ \
\ \ \ \ \ \ \ \ \ \ \ \ \ \ }   \\
\frac{\partial W}{\partial t}-c\Delta W &=&D_{3}U\ \ \ +\left( -\left( \beta
+1+D_{3}\right) +2wz\right) W\ \ \ \ \ \ +w^{2}Z,\ \ \ \ \ \ \ \ \ \ \ \ \ \
\ \ \    \\
\frac{\partial Z}{\partial t}-d\Delta Z &=&\ D_{4}V\ \ \ \ +\left( \beta
-2wz\right) W-\left( D_{4}+w^{2}\right) Z.\text{\ }\ \ \ \ \ \ \ \ \ \ \ \ \
\ \ \ \ \ \ \ \ \ \ \ \ \ \ \ \ \ \   
\end{eqnarray}%
Replace in the linearized system, we get%
\begin{eqnarray}
\frac{\partial U}{\partial t}-a\Delta U &=&-\left( -\beta +1+D_{1}\right) U\
+\alpha ^{2}V\ \ \ +D_{1}W,\text{ \ \ \ }  \\
\frac{\partial V}{\partial t}-b\Delta V &=&-\beta U\ -\left( D_{2}+\alpha
^{2}\right) V\ +D_{2}Z,\text{\ \ \ \ \ \ \ \ \ \ \ \ \ \ \ \ \ \ \ \ \ \ \ \ 
}   \\
\frac{\partial W}{\partial t}-c\Delta W &=&\ \ D_{3}U\ -\left( -\beta
+1+D_{3}\right) W\ \ \ \ \ \ +\alpha ^{2}Z,\ \ \ \ \ \ \ \ \ \ \ \  
 \\
\frac{\partial Z}{\partial t}-d\Delta Z &=&\ D_{4}V\ \ -\beta W-\left(
D_{4}+\alpha ^{2}\right) Z,\text{\ \ \ }\ \ \ \ \ \ \ \ \ \ \ \ \ \ \ \ \ \
\ \ \ \  
\end{eqnarray}%
The generator $B$ of the linear semigroup $L(t)$%
\begin{equation*}
B=\left( 
\begin{array}{cccc}
a\Delta -\left( -\beta +1+D_{1}\right)  & \alpha ^{2} & D_{1} & 0 \\ 
-\beta  & b\Delta -\left( D_{2}+\alpha ^{2}\right)  & 0 & D_{2} \\ 
D_{3} & 0 & c\Delta -\left( -\beta +1+D_{3}\right)  & \alpha ^{2} \\ 
0 & \ D_{4} & -\beta  & d\Delta -\left( D_{4}+\alpha ^{2}\right) 
\end{array}%
\right) .
\end{equation*}%
Let $\mu _{1}=0<\mu _{2}<...<\mu _{n}<...$ the eigenvalues of $-\Delta $
with the homogeneous Neumann boundary and $\varphi _{1},\ \varphi
_{2},...,\varphi _{n},...$ their associated eigenvectors, we try to find the
eigenvectors of the operator $B$ on the form $\Phi _{i}=\left(
p_{i},q_{i},r_{i},s_{i}\right) \varphi _{i},\ i=1,...,n,...$, with
associated eigenvalues $\lambda _{1},\ \lambda _{2},...,\lambda _{n},...$,
then we have%
\begin{equation*}
\left( B-\lambda _{i}I\right) \Phi _{i}=0,
\end{equation*}%
so%
\begin{eqnarray*} 
\left\vert 
\begin{array}{cccc}
-a\mu _{i}-\left( -\beta +1+D_{1}\right) -\lambda _{i} & \alpha ^{2} & D_{1}
& 0 \\ 
-\beta  & -b\mu _{i}-\left( D_{2}+\alpha ^{2}\right) -\lambda _{i} & 0 & 
D_{2} \\ 
D_{3} & 0 & -c\mu _{i}-\left( -\beta +1+D_{3}\right) -\lambda _{i} & \alpha
^{2} \\ 
0 & \ D_{4} & -\beta  & -d\mu _{i}-\left( D_{4}+\alpha ^{2}\right) -\lambda
_{i} \\
\end{array}%
\right\vert  \\
\end{eqnarray*}%
$= 0 $.

The principal determinant of this algebraic linear system is a fourth degree
polynomial of unknown $\lambda _{i}.$ It admits a root with no positive real
part if the real part of the sum of its four roots is not positive, that is%
\begin{equation*}
-\left( a+b+c+d\right) \emph{Re}\mu _{i}+\left[ 2\left( \beta -1-\alpha
^{2}\right) -\left( D_{1}+D_{2}+D_{3}+D_{4}\right) \right] >0
\end{equation*}%
Since%
\begin{equation*}
\mu _{i}\simeq Ki^{\frac{2}{N}}
\end{equation*}%
then%
\begin{equation*}
\frac{\left[ 2\left( \beta -1-\alpha ^{2}\right) -\left(
D_{1}+D_{2}+D_{3}+D_{4}\right) \right] }{\left( a+b+c+d\right) }\simeq Ki^{%
\frac{2}{N}}
\end{equation*}%
Thus if%
\begin{equation*}
\left( D_{1}+D_{2}+D_{3}+D_{4}\right) <2\left( \beta -1-\alpha ^{2}\right) 
\end{equation*}%

We can state the following result

\begin{theorem}
\label{gattrdlower}
Consider the reaction diffusion equation described via, \eqref{(1.1)}-\eqref{(1.5)}. The global attractor $\mathcal{A}$ of the system, for spatial dimension $N \leq 3$, has explicit lower bounds for its Hausdorff dimension, in particular 
there exists a universal constant $K^{\prime }$ such that 

\begin{equation*}
\dim _{H}A\geq K^{\prime }\left\{ \frac{\left[ 2\left( \beta -1-\alpha
^{2}\right) -\left( D_{1}+D_{2}+D_{3}+D_{4}\right) \right] }{\left(
a+b+c+d\right) }\right\} ^{\frac{N}{2}}.
\end{equation*}

\end{theorem}

\section{Numerical Simulations}
\label{8}

We now carry out numerical simulations of \eqref{(1.1)}-\eqref{(1.5)}. 
 In order to explore the dynamics of the model in 2d, we use a finite difference method. A forward difference scheme is 
used for the reaction terms. For the diffusion terms, a standard five point explicit finite difference scheme is 
used. The numerical simulation is carried out at different time 
levels for two dimensional spatial model system. The system of equations is numerically solved over 200 $\times$ 200 mesh points , on a domain of size $L_x \times L_y$, where $L_x=L_y=500$.
with spatial resolution  $\Delta x$ = $\Delta y$ = 1 and time step  $\Delta t =  \frac{1}{24}$. The initial condition used is a small 
perturbation about $(2, 2.75, 2, 2.75)$ and the boundary conditions used are no flux Neumann conditions. Note this is fine, as the results hold for Neumann boundary conditions as well.

\begin{figure}[!ht]
	\begin{center}
		\includegraphics[width=0.35\textwidth]{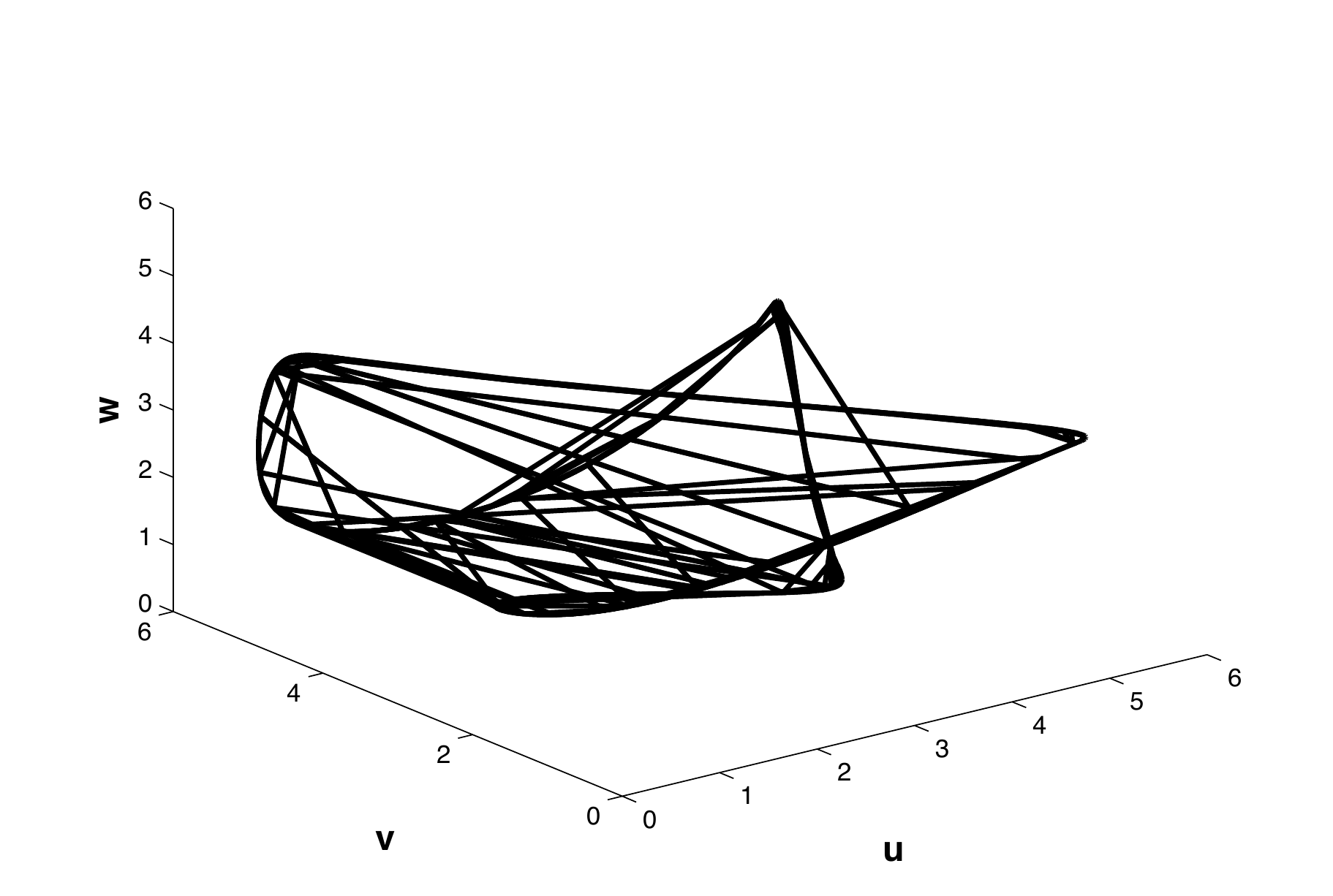}
			\includegraphics[width=0.35\textwidth]{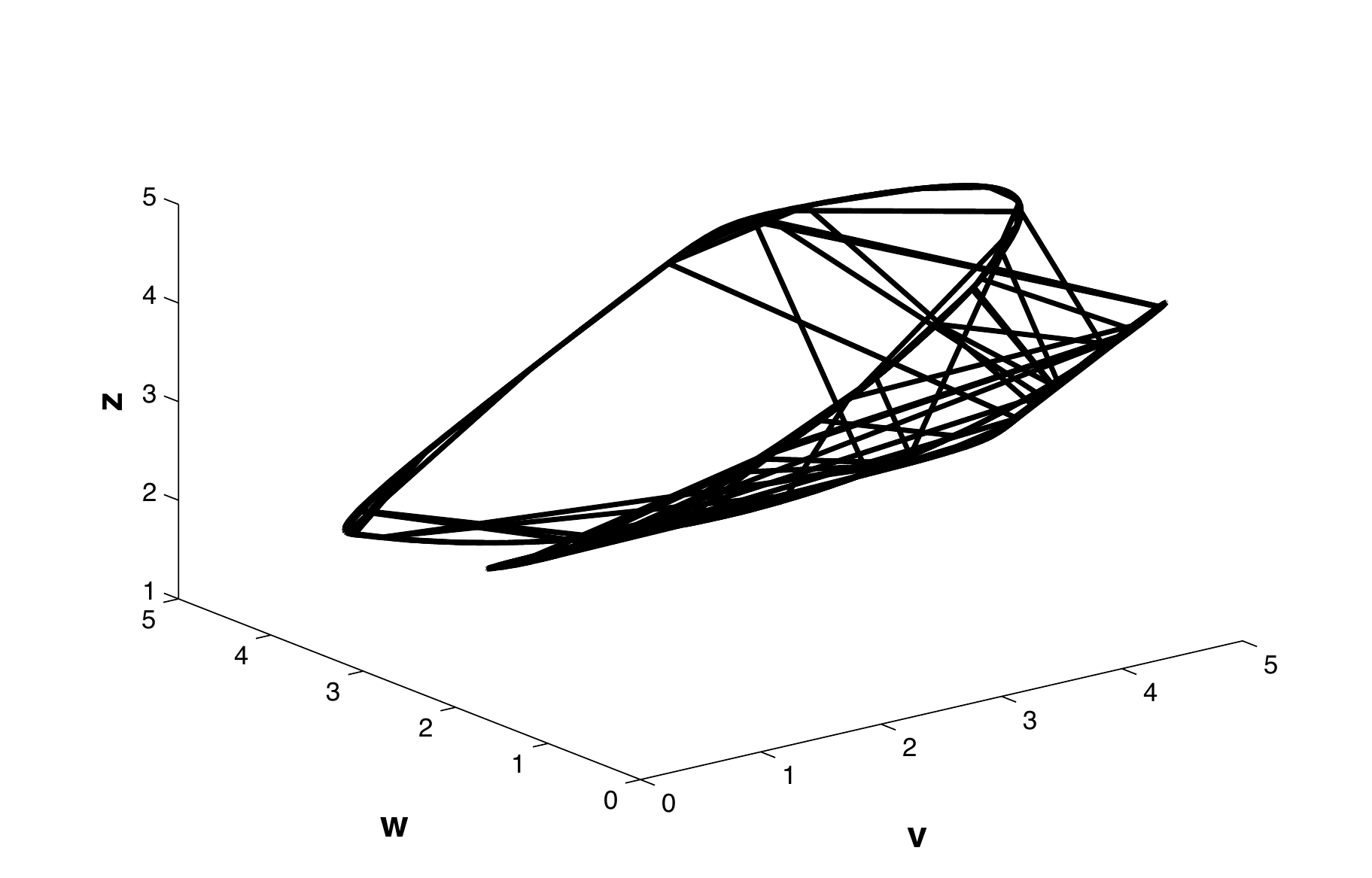}
			\includegraphics[width=0.35\textwidth]{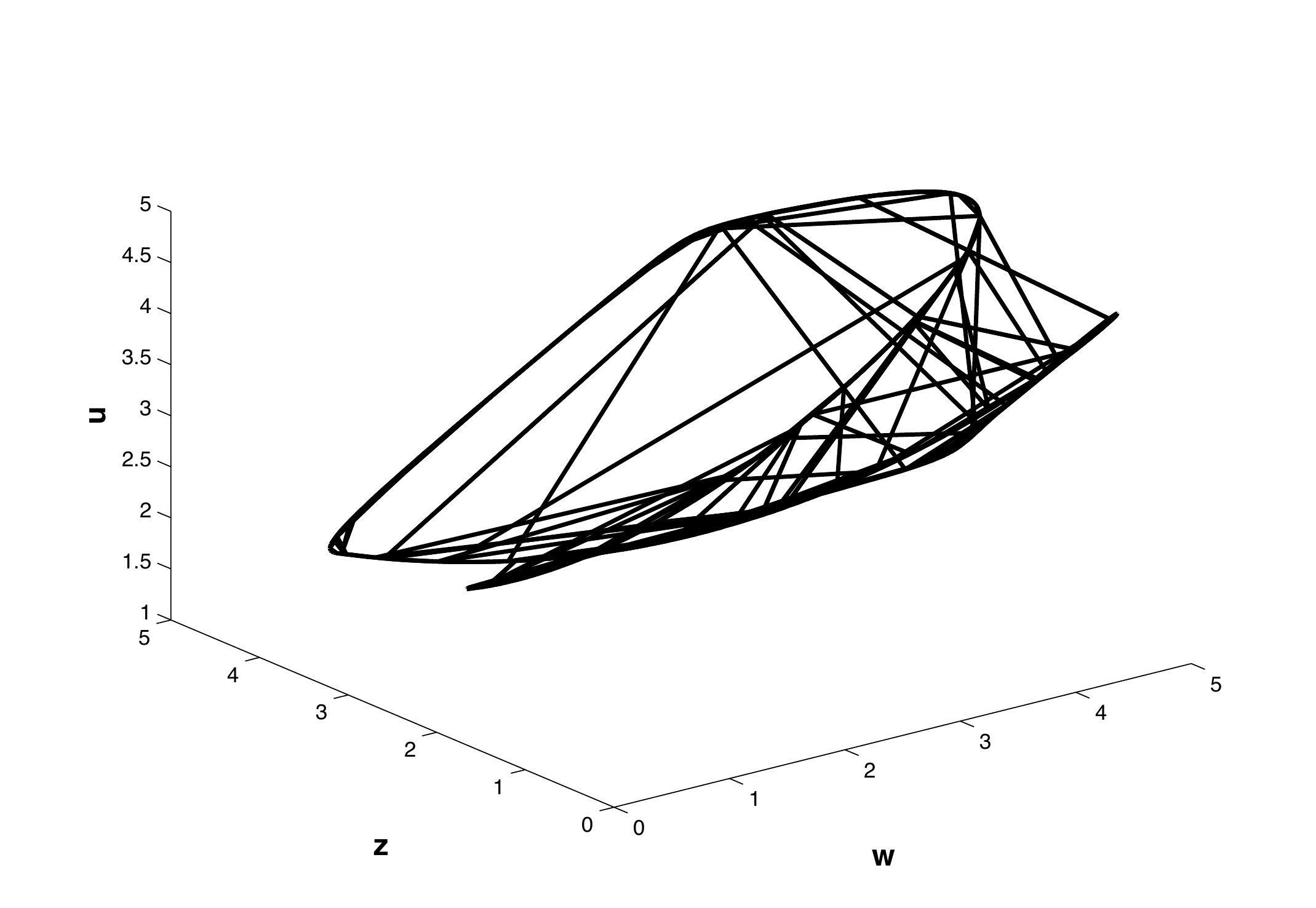}
			\end{center}
		\caption{Attarctor for system \eqref{(1.1)}-\eqref{(1.5)} in one dimensional domain obtained at $t=10000$.}
	\label{fig:II}
\end{figure}

\begin{table}[ht]
\caption{Parameters used in the simulations in fig 1,2,3,4 }
\centering  
\begin{tabular}{ c | c | c | c | c | c | c | c | c | c}
\hline\hline                        
 $\alpha$ &  $\beta$ & $D_1$ & $D_2$ & $D_3$ & $D_4$ & $a$ & $b$ & $c$ & $d$ \\ [0.5ex] 
\hline   
  2 & 5.5 & 0.0126 & 0.126 & 0.0125  &0.125  & $10^{-6}$ & $10^{-6}$ & $10^{-6}$ & $10^{-6}$  \\
\hline 
\end{tabular}
\label{table:parameters} 
\end{table}

\begin{figure}[!ht]
	\begin{center}
		\includegraphics[width=0.6\textwidth]{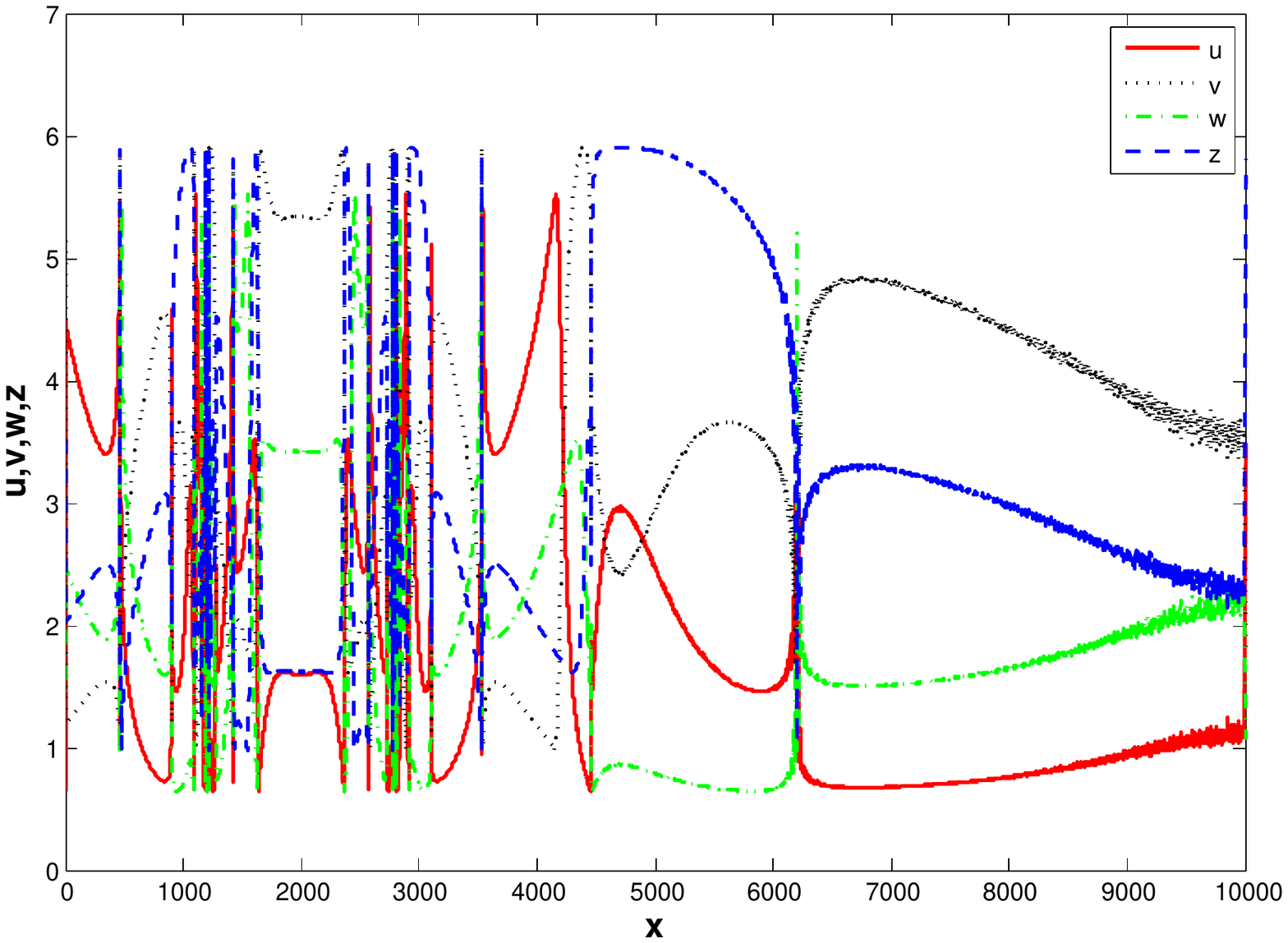}
		\end{center}
		\caption{Spatial distribution of $u,v,w,z$ species at $t=10000$.}
	\label{fig:I}
\end{figure}
%


\begin{figure}[!htp]
	\begin{center}
	\includegraphics[scale=0.35]{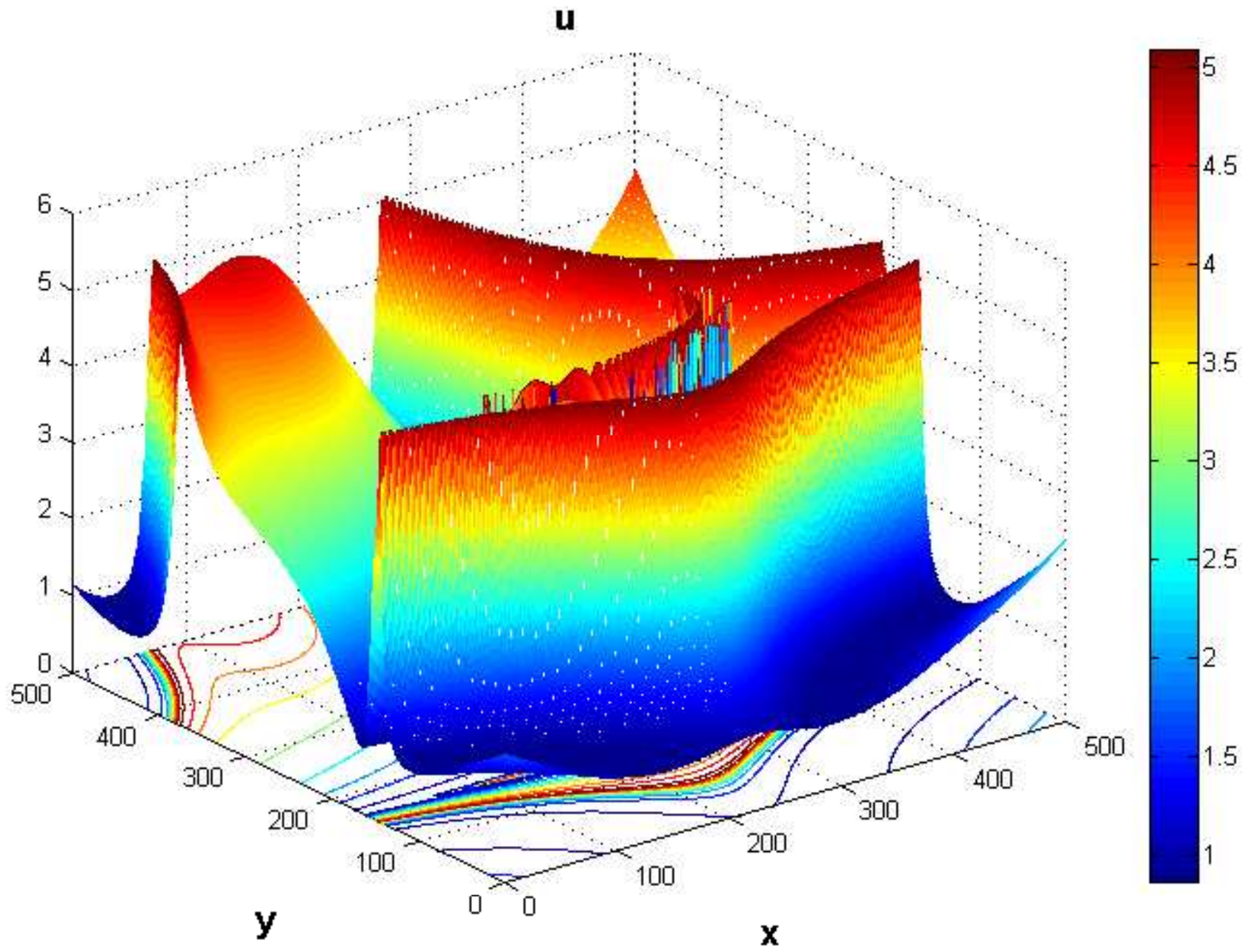}
			\includegraphics[scale=0.35]{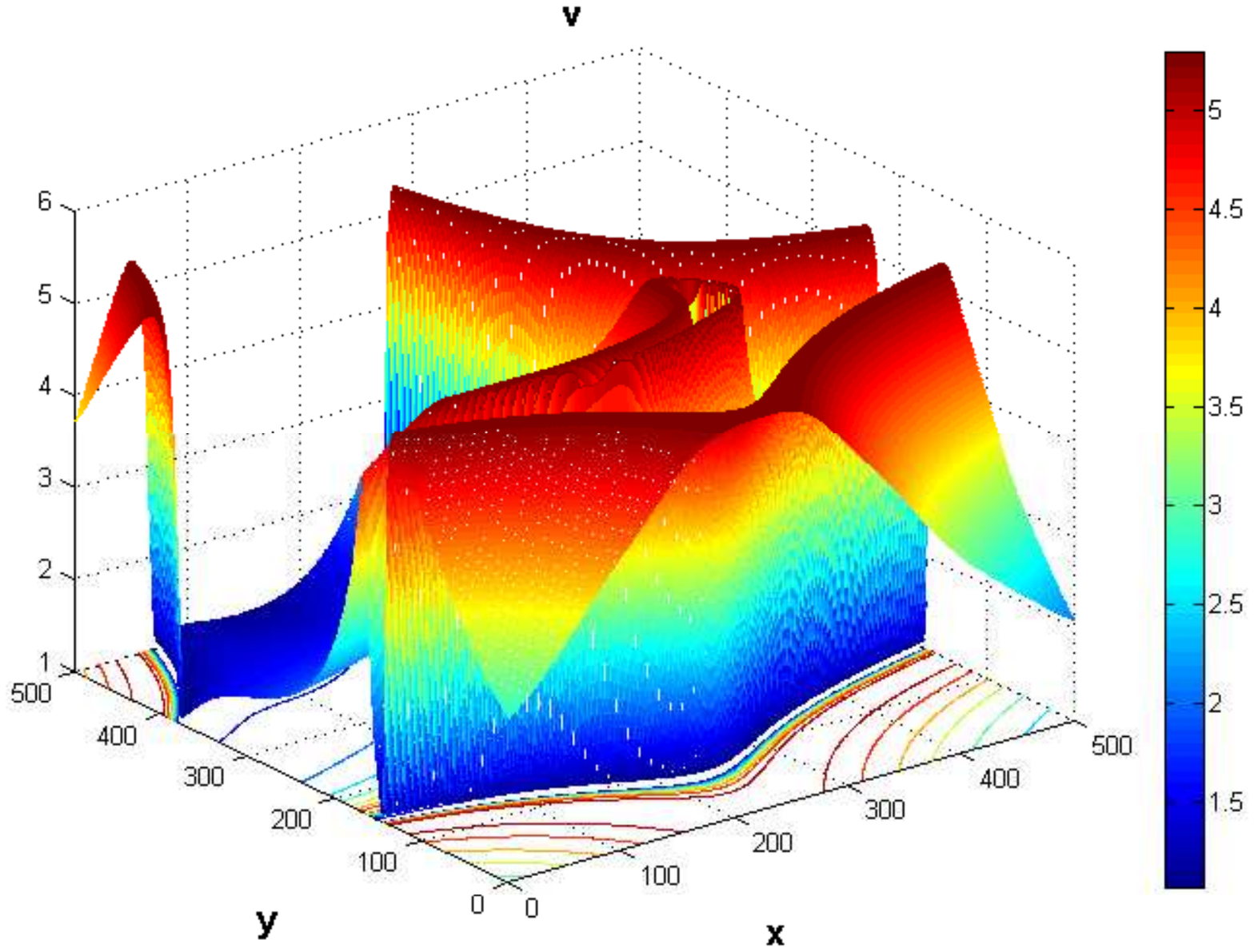}
			\includegraphics[scale=0.35]{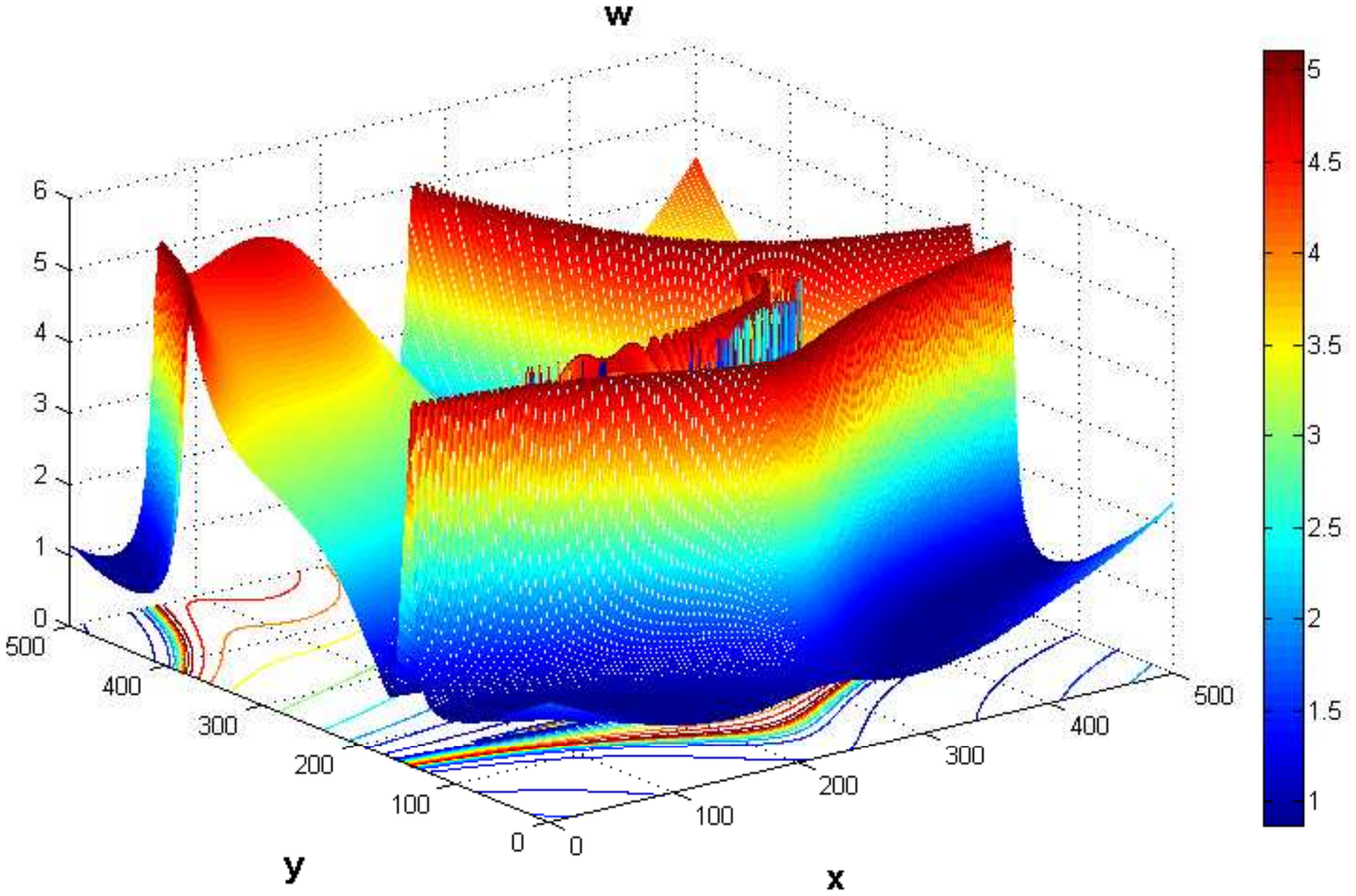}
			\includegraphics[scale=0.35]{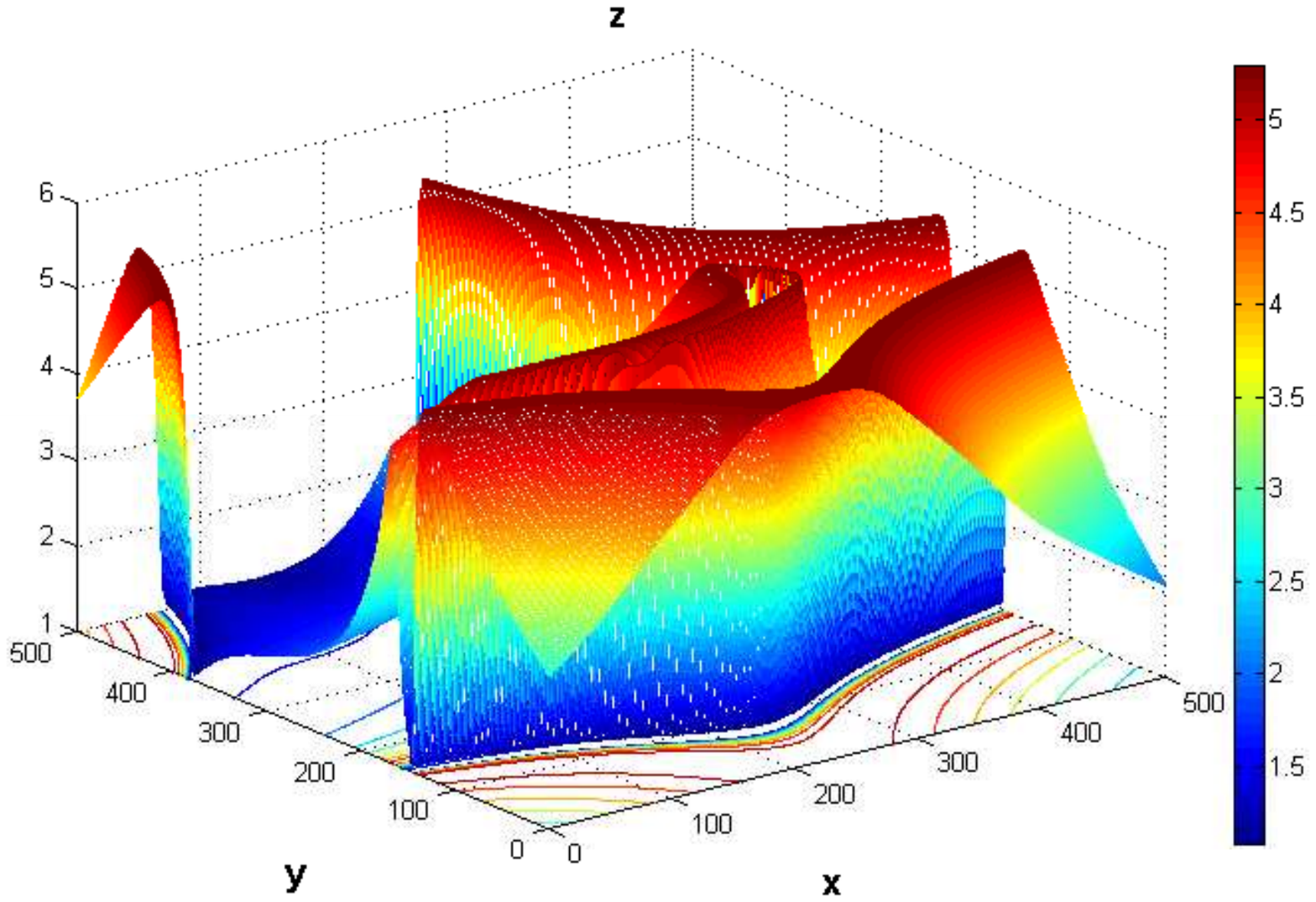}
		 \end{center}
		\caption{ The two dimensional spatial attractor of the system  \eqref{(1.1)}-\eqref{(1.5)} at $t=50$. }
		\end{figure}

\begin{figure}[!htp]
	\begin{center}
		\includegraphics[scale=0.35]{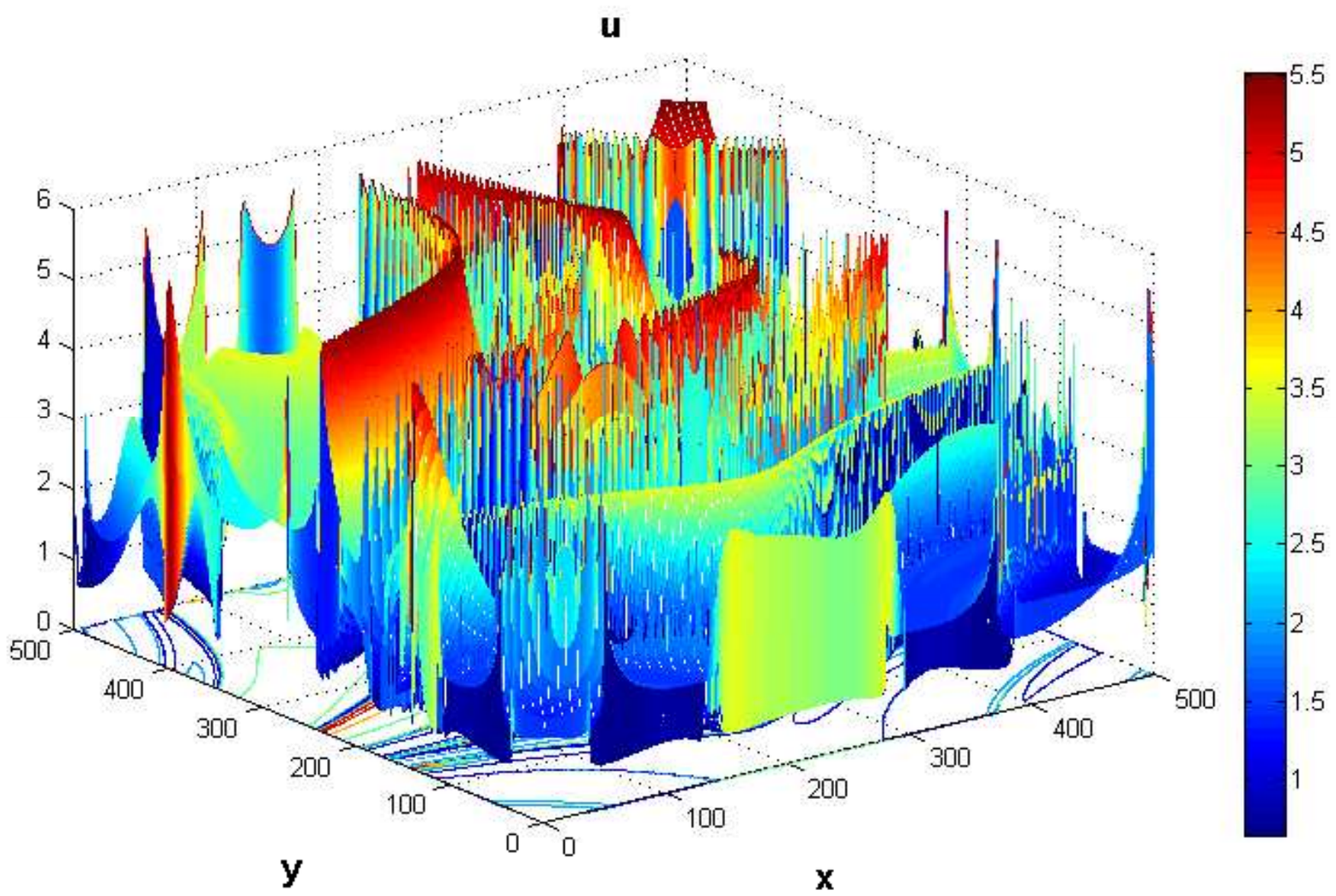}
			\includegraphics[scale=0.35]{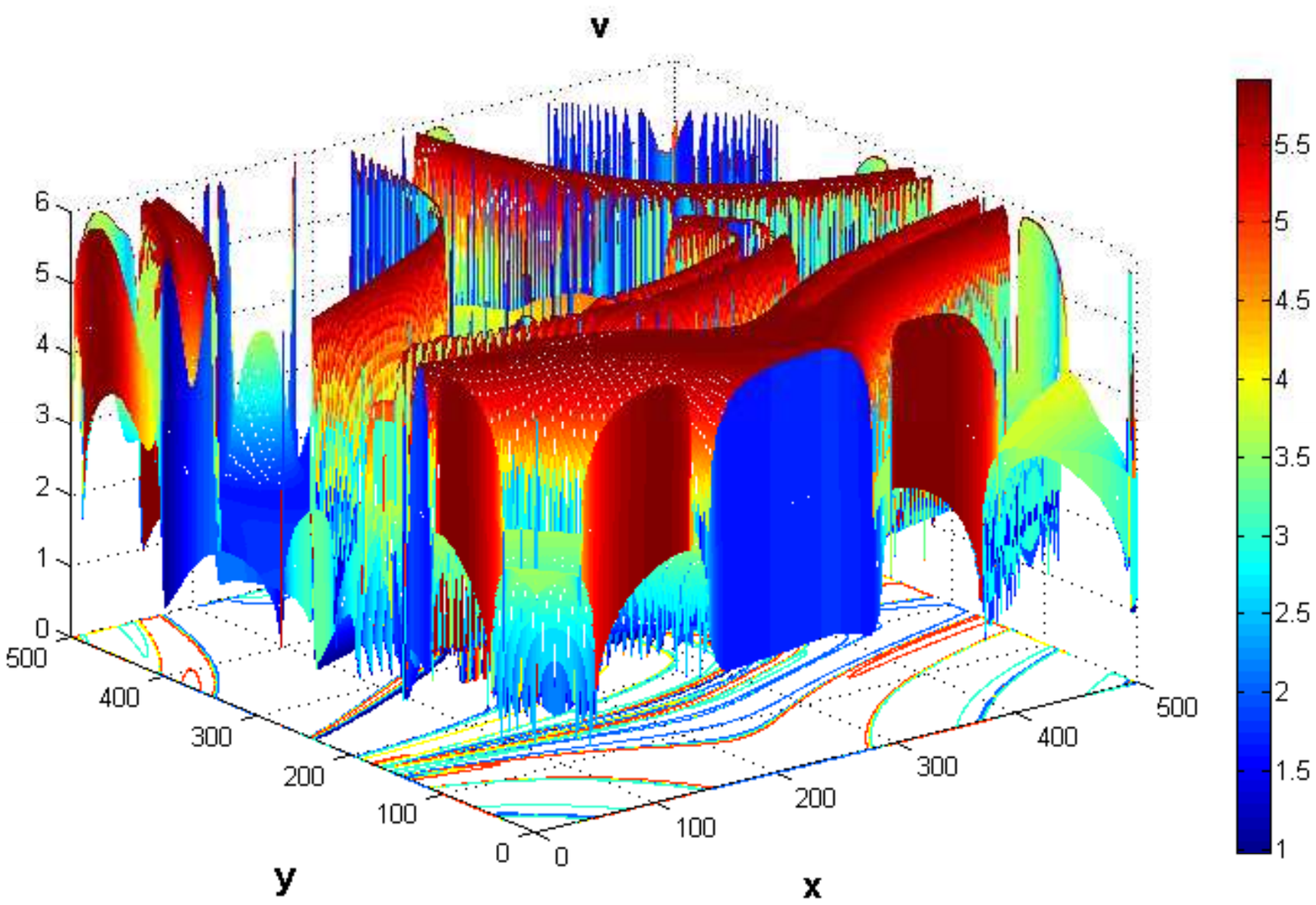}
			\includegraphics[scale=0.35]{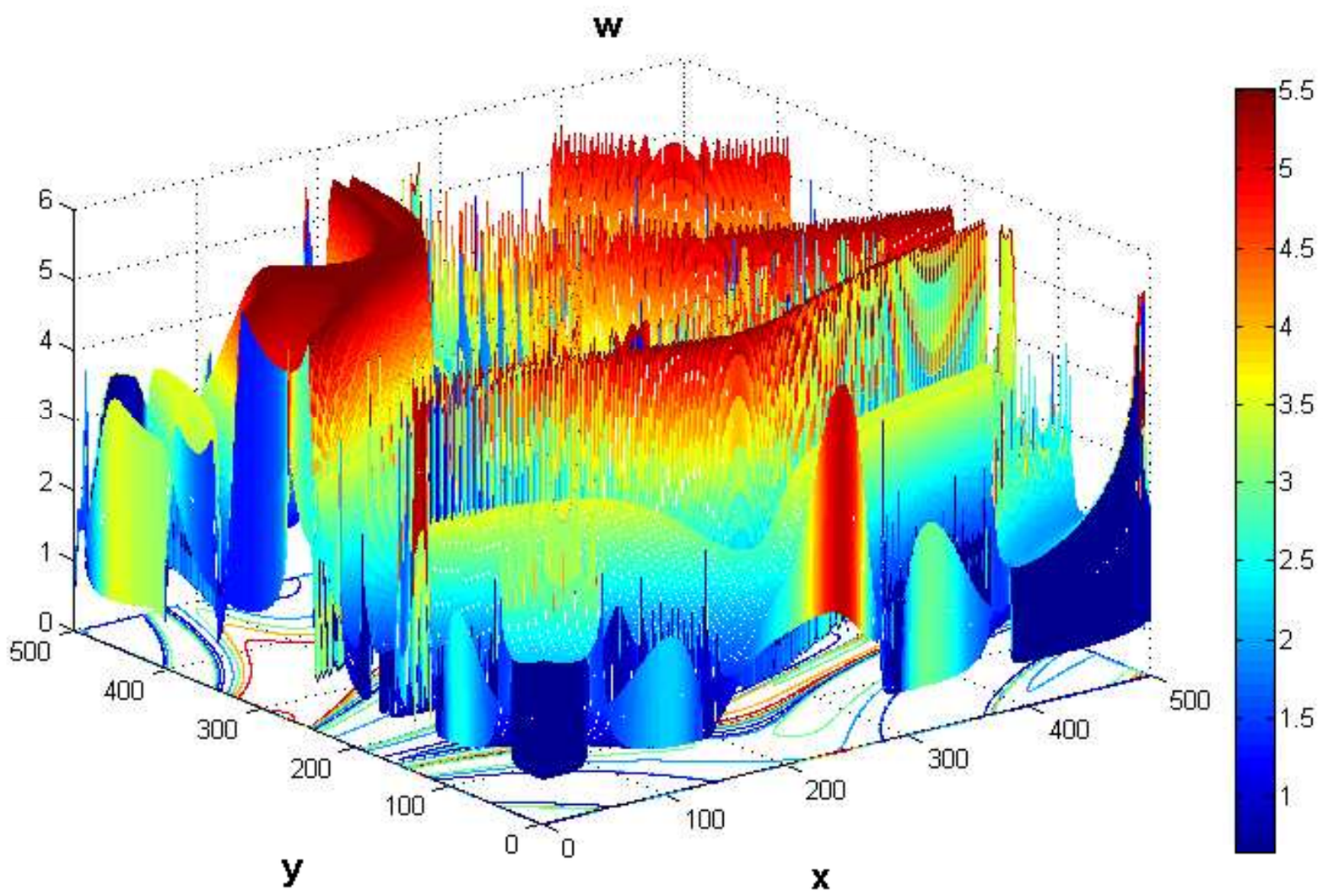}
			\includegraphics[scale=0.35]{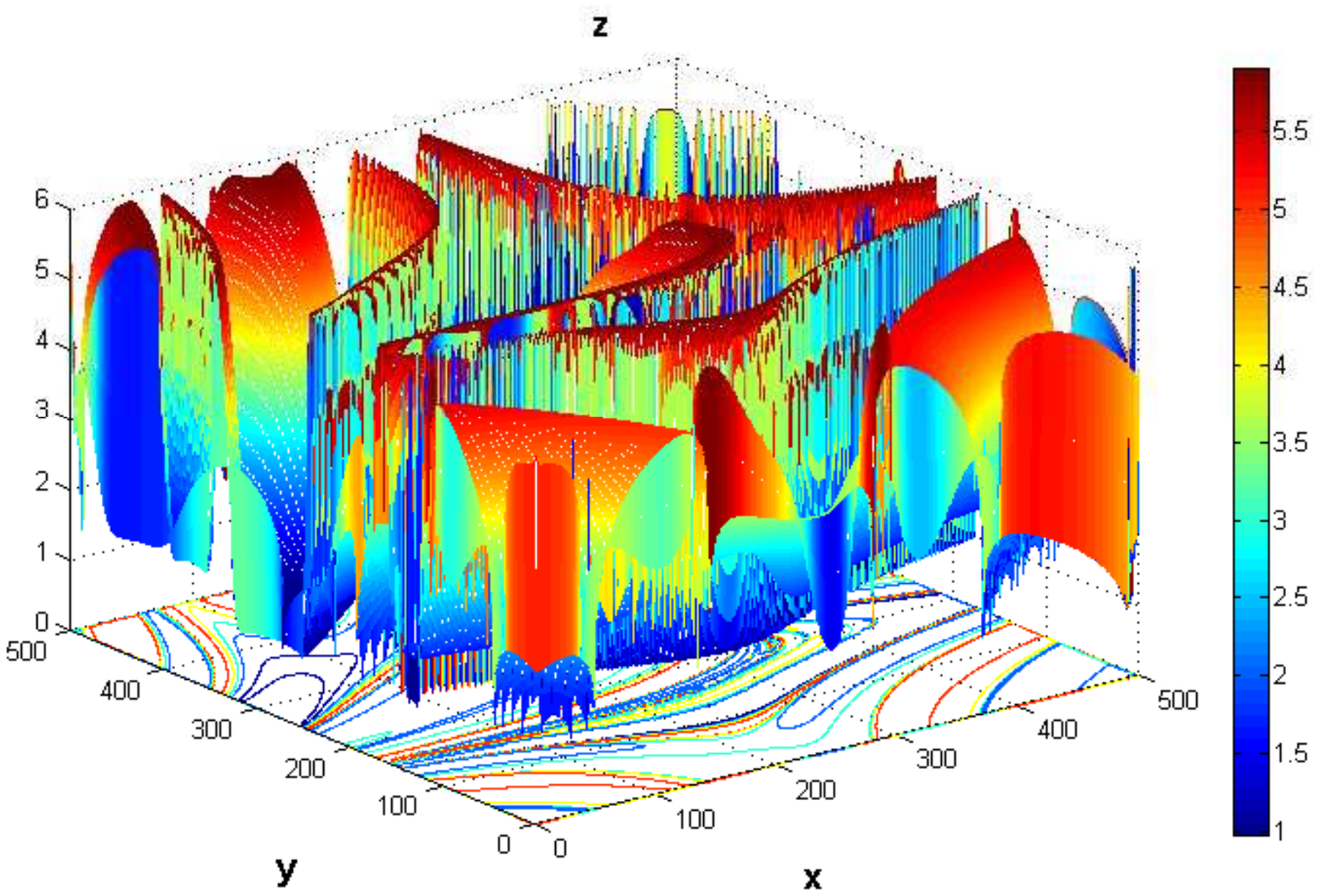}
                       	\end{center}
		\caption{ The two dimensional spatial attractor of the system  \eqref{(1.1)}-\eqref{(1.5)} at $t=500$. }
		\end{figure}

\begin{figure}[!htp]
	\begin{center}
	             \includegraphics[scale=0.35]{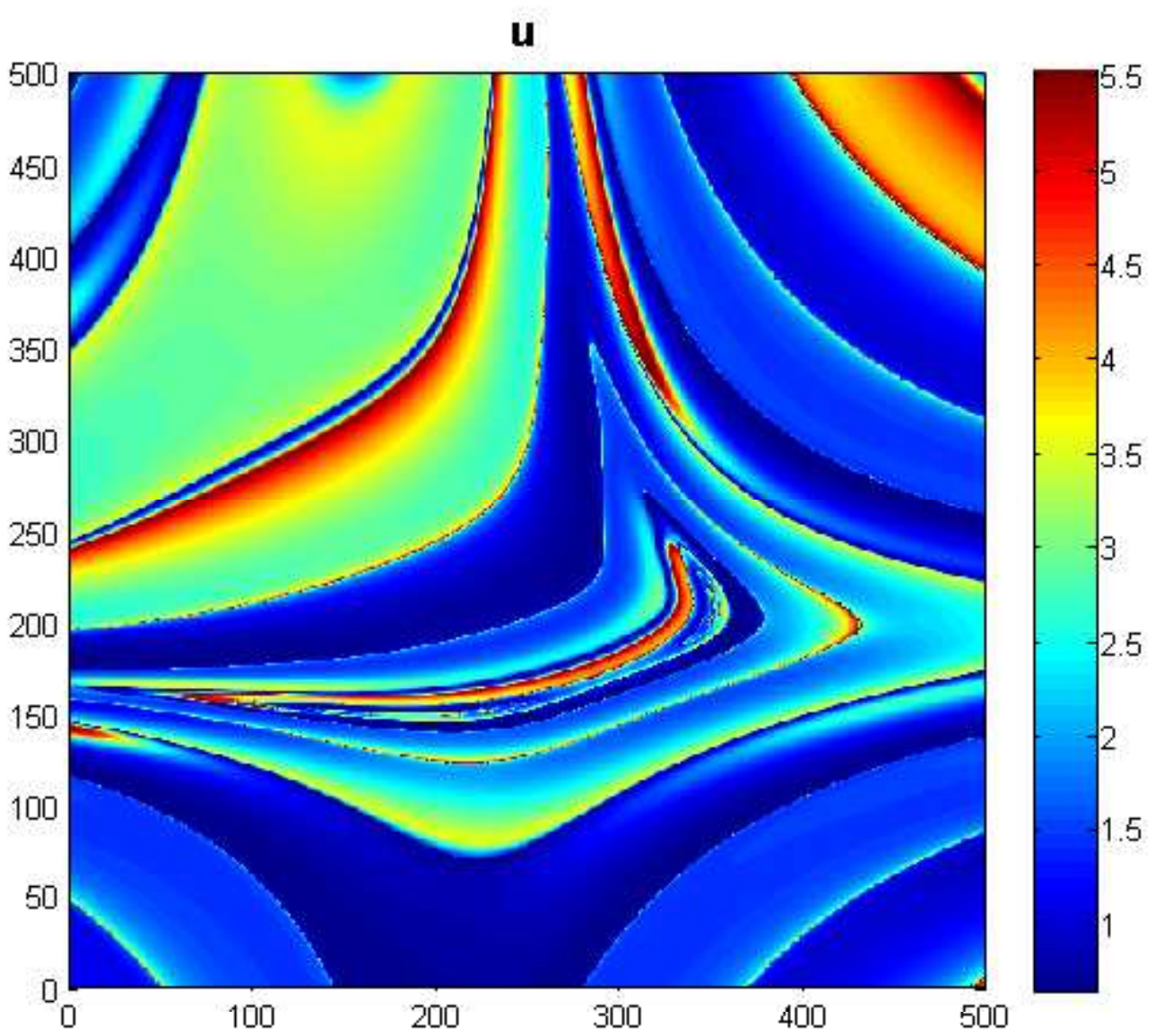}
			\includegraphics[scale=0.35]{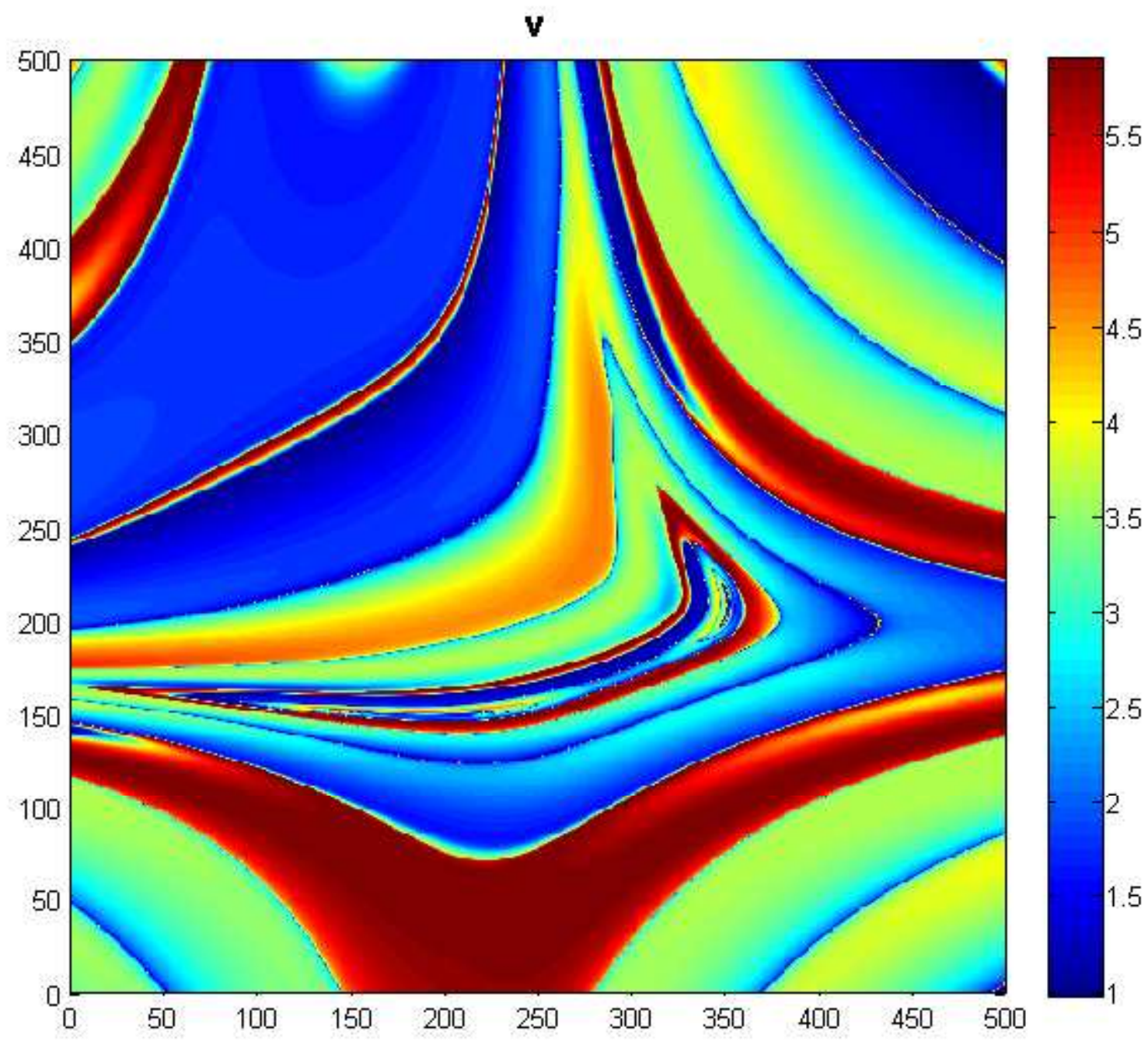}
			\includegraphics[scale=0.35]{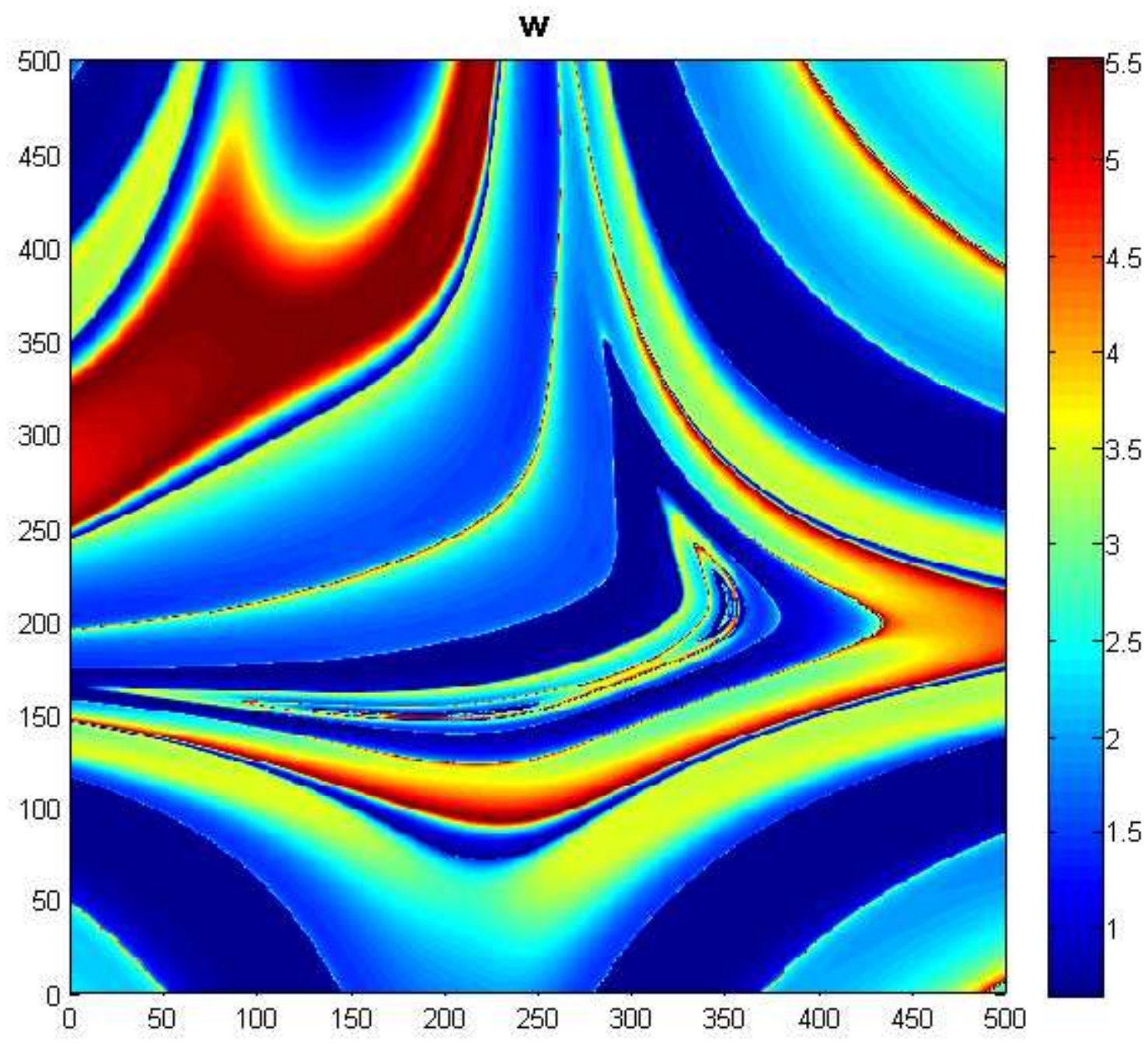}
			\includegraphics[scale=0.35]{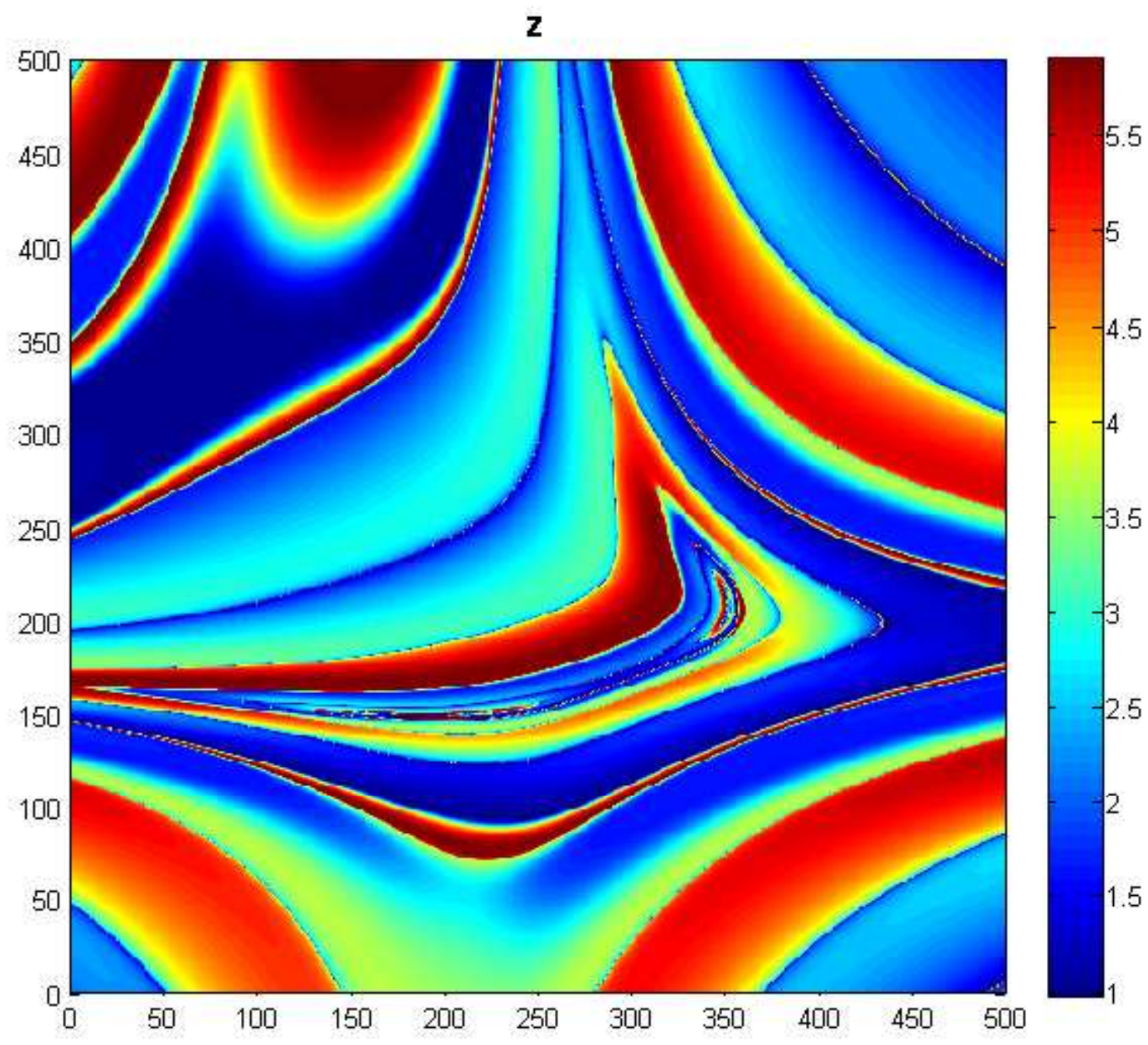}
			\end{center}
		\caption{The densities of $u,v,w,z$ species are shown as contour plots in two dimensional plane at $t=500$. }
		\end{figure}

\section{Attractor reconstruction}
\label{9}

The dynamics of the four-species Brusselator model described by \eqref{(1.1)} - \eqref{(1.5)} is visualized with the aid of a delay-time method. This method was proposed by Takens and Mane \cite{TM81}, and it ensures equivalence between the topological properties of the actual and reconstructed attractors. The delay method consists of a reconstruction performed from measurements of a single variable. Here we  consider the variable $u$, which is one of the four state variables of the system. 

According to Takens and Mane \cite{TM81}, the dynamics of the system can be completely described by the time series, $u(t)$, which is numerically calculated  at  a fixed position in space. This time series can be embedded  in an $m$-dimensional embedding space and reconstructed as a \emph{pseudo-trajectory } through the following embedding vectors, which form an embedding matrix:  
\begin{eqnarray}
&&y_1 = (x(t_0 ), x(t_0 + \tau ),\dots, x(t_0 + (m - 1)\tau ))^{T},
\nonumber \\
&&y_2 = (x(t_0 + l ), x(t_0 + l + \tau )), \dots, x(t_0 + l +(m-1)\tau)^{T},
\nonumber \\
&& \dots
\nonumber \\
&&y_s = (x(t_0 + (s - 1)l ), x(t_0 + (s - 1)l + \tau), \dots, x(t_0 + (s - 1)l + (m - 1)\tau ))^{T}. \nonumber \\
\end{eqnarray}
Here, $\tau$ is called the ``delay-time", $l$ is the sampling interval, and $w=(m-1)\tau$  is the ``window length"  which represents the time spanned by each embedding vector. Selection of appropriate values for parameters $\tau$ and $l$ in the embedding procedure is particularly important for the reliability of the results. To ensure the equivalence between the topological properties of actual and reconstructed attractors, a formal criterion was proposed by Takens, namely that $m\geq 2d+1$.  This criterion relates the embedding dimension, $m$, to the attractor dimension, $d$. In practice, the dimension, $d$, of the attractor is unknown and has to be determined. Choosing the optimal embedding parameter, $m$, and the delay-time, $\tau$, is rather nontrivial. 

In this paper, we adopt the interactive technique suggested by Albano \textit{et al.} \cite{AM88}. The authors have combined a singular-value decomposition, which leads to a set of statistically independent variables, and the Grassberger-Procaccia algorithm, \cite{GP84}, to determine the dimension of the attractors. The algorithm consists of four steps: 

\begin{enumerate}

\item Choose $m$  and $\tau$ so that the window, $w$, is a few times larger than the correlation time of the time series, $u(t)$. A  rule thumb for selecting $\tau$ is to choose the time at which the autocorrelation time falls to about $1/e$. This delay-time ensures that the embedding vectors  spanning the phase space have become, in some sense, independent. 

\item Perform a singular-value decomposition of the embedding matrix, $Y =V\Sigma U^T$. Singular values, $\sigma_i$, occupying the diagonal of the diagonal matrix, $\Sigma$, are discarded below a certain threshold, because one can  consider them as originating from noise.  In this paper,  the threshold was set to be $10^{-2}$. 

\item Calculate the correlation integral from which  the correlation dimension, $d$, of the attractor is then deduced as the limit of the slopes of the log-log plot of the correlation integrals. A test of Takens' criterion is then performed; if it is not satisfied, i.e., if $m\le2d+1$, then  increase $m$ until this criterion is satisfied.  

\item Refine the value of the embedding dimension, $m$; the suitable embedding dimension is the one that maximizes the straight line parts in the log-log plot of the correlation integrals.

\end{enumerate}

It is well known that the addition of diffusion, to an ODE system, can lead to temporal chaos. In fact, the model without diffusion (the ODE case), exhibiting oscillatory dynamics for a certain parameter set, can be destabilized to chaos solely via adding diffusion \cite{Pa93}. Here, we test this observation on \eqref{(1.1)} - \eqref{(1.5)}.  We start our numerical experiments from a periodic state that exists in the absence of diffusion. The following  parameter set yields a periodic state: 

\begin{equation}   
D_{1} = .0126, D_{2} = .126, D_{3} = .0125, D_{4} = .125, \alpha = 2, \beta = 5.9. 
\end{equation}

This oscillatory dynamics is described by a limit cycle. Imposing diffusion in the four-species Brusselator model, that is

\begin{equation}
 a = 10^{-6}, b = 2 \times 10^{-6}, c = 3 \times 10^{-6}, d = 4 \times 10^{-6},
 \end{equation}
 
shows that the oscillatory dynamics becomes chaotic. We next reconstruct the low dimensional attractor from the time series of species $u$, which is obtained by fixing a spatial location and following the trajectory in time. The dimension of the reconstructed attractor, which corresponds to the slope of the correlation dimension of the attractor, is found to be approximately $27.54$. This shows that the attractor of the full diffusive system is quite high.  

For the above range of parameters, we can extract a lower bound for the constant $K^{\prime}$, from Theorem \ref{gattrdlower}. As 
  
 \begin{equation}
  K^{\prime }\left\{ \frac{\left[ 2\left( \beta -1-\alpha
^{2}\right) -\left( D_{1}+D_{2}+D_{3}+D_{4}\right) \right] }{\left(
a+b+c+d\right) }\right\} ^{\frac{N}{2}} \geq 27.54,
\end{equation}
 
one obtains after plugging in the parameters that essentially 

 \begin{equation}
 K^{\prime} \geq \frac{91}{100}.
 \end
{equation}

Note that the fractional dimension of $27.54$, indicates that the attractor is strange. 
To test whether or not the dynamics on it is chaotic, we estimated the Lyapunov exponents by the method proposed by Lai and Chen \cite{LC98}. We find the largest exponent to be approximately equal to $2.17$. This implies that the dynamics on the strange attractor, for this parameter set, is indeed chaotic in time. In particular, this also tells us that for a a parameter set where $D_1 \neq D_3$, $D_2 \neq D_4$, one may obtain a chaotic attractor, the existence of which becomes harder to prove, although possible, as we demonstrate.

\end{document}